%% file: CAT.tex
\newif\ifSIAM
\SIAMfalse

\newif\ifARXIV
\ARXIVtrue

\newif\ifARXIVorSIAM
\ifSIAM
\ARXIVorSIAMtrue
\fi
\ifARXIV
\ARXIVorSIAMtrue
\fi

\ifARXIV
\documentclass{article}
\usepackage{graphicx,fancyhdr,amsmath,amssymb,url}
\usepackage{amsthm}
\fi
\ifSIAM
\documentclass[review,hidelinks,onefignum,onetabnum]{siamart250211}
\fi

\input{CAT_shared}

\newcommand{\MyThanks}[0]{Department of Industrial Engineering, University of Pittsburgh. Email: \{fah33, ohinder\}@pitt.edu.}

\ifARXIVorSIAM
\author{Fadi Hamad, Oliver Hinder\thanks{\MyThanks}}
\def\cite{\citep}
\fi

\begin{document}

\maketitle
\begin{abstract}
We present an adaptive trust-region method for unconstrained optimization that allows inexact solutions to the trust-region subproblems. Our method is a simple variant of the classical trust-region method of \citet{sorensen1982newton}. The method achieves the best possible convergence bound up to an additive logarithmic term for finding an $\epsilon$-approximate stationary point, i.e., $O( \Delta_f L^{1/2}  \epsilon^{-3/2}) + \tilde{O}(1)$ iterations, where $L$ is the Lipschitz constant of the Hessian, $\Delta_f$ is the optimality gap, and $\epsilon$ is the termination tolerance for the gradient norm. This improves over existing trust-region methods whose worst-case bound is at least a factor of $L$ worse. We compare our performance with state-of-the-art trust-region (TRU) and cubic regularization (ARC) methods from the GALAHAD library on the CUTEst benchmark problems with \editTwo{at least} 100 variables. \editTwo{We also compare with the recently developed Universal trust-region (UTR) method, using the same subproblem solver.
In terms of shifted geometric mean of wall-clock times our 
method's is between $1.2\times$ and $2\times$ faster.
We report similar improvements for number of function evaluations, factorizations, gradient evaluations and Hessian evaluations. 
}

Compared to the conference version of this paper \cite{hamad2022consistently}, our revised method includes several practical enhancements. These modifications dramatically improved performance, including almost an order of magnitude reduction in the shifted geometric mean of wall-clock times. \edit{We also show that it suffices for \editTwo{the function to be continuously twice-differentiable} to guarantee that either the minimum gradient norm converges to zero or the objective value tends towards negative infinity, even when the iterates diverge.}
\end{abstract}

\ifSIAM

\fi
\section{Introduction}

Consider the unconstrained optimization problem
\edit{\[
\min_{x \in \R^n} f(x)
\]}
where the function $f: \R^ n \rightarrow \R$ is twice differentiable and possibly non-convex. Finding a global minimizer of this problem is intractable for large $n$ \cite[sect. 1.6]{nemirovskii1983problem}; instead, we aim to find an $\epsilon$-approximate stationary point, i.e., a point $x$ with $\| \grad f(x) \| \le \epsilon$ where $\epsilon > 0$. Second-order methods \cite{sorensen1982newton,curtis2017trust,curtis2021trust,curtis2023worst,jiang2023universal, leconte2023complexity,leconte2024interior,nesterov2006cubic,cartis2011adaptiveII,royer2018complexity, liu2022newton, he2023newton}, which choose their next iterates based on gradient and Hessian information at the current iterate, are popular for solving these problems. For example, Newton's method, one of the earliest and simplest second-order methods, produces its search directions $d_k$ by minimizing the second-order Taylor series approximation of the function at the current iterate:
\begin{flalign}\label{eq:newton-direction}
	d_k  \in \argmin_{d} M_k(d) := \frac{1}{2} d^\top \grad^2 f(x_k) d +\grad f(x_k)^\top d,
\end{flalign}
and the next iterate becomes $x_{k+1} = x_k + d_k$.
Unfortunately, unless $\grad^2 f(x_k)$ is positive definite, this calculation is not well-defined.  
To rectify this issue, \citet{sorensen1982newton} proposed restricting the search direction 
in \eqref{eq:newton-direction} to a ball of radius $r_k$ around the current iterate:
\begin{flalign}\label{eq:trust-region-subproblem}
	d_k  \in \argmin_{d : \| d \| \le r_k} M_k(d)
\end{flalign}
where $\| \cdot \|$ is the spectral norm for matrices and the Euclidean norm for vectors. 
This approach, known as a trust-region method \cite{conn2000trust}, is reliable and popular. \citet[Theorem 4.5]{sorensen1982newton} shows that its iterates asymptotically converge to stationary points, i.e., points $x$ with $\grad f(x) = 0$.
On the other hand, while knowing that the iterates will eventually converge to an approximate stationary point is desirable, it leaves open the question: how long could this take?

This question motivates the study of worst-case convergence bounds for finding approximate stationary points.
Under the assumption that the gradient is smooth, i.e., there exists $S > 0$ such that $\| \grad f(u) - \grad f(v) \| \le S \| u - v \|$ for all $u, v \in \R^{n}$ and the optimality gap $\Delta_f := f(x_1) - \inf_{x \in \R^{n}} f(x)$ at the starting point $x_1$ is finite, gradient descent requires at most $2 S \Delta_f  \epsilon^{-2}$ iterations
to find an $\epsilon$-approximate stationary point \cite[Chap. 1]{nesterov2013introductory}.  This is the best possible worst-case bound (up to constant factors) on smooth functions \cite{carmon2017loweri}.
On the other hand, Nesterov and Polyak \citep{nesterov2006cubic} developed a second-order method that on functions with $L$-Lipschitz Hessian, i.e.,
\[ \| \grad^2 f(u) - \grad^2 f(v) \| \le L \| u - v \| \quad \forall u, v \in \R^{n}, \]
requires at most $O(\Delta_f L^{1/2} \epsilon^{-3/2} )$ iterations to find an $\epsilon$-approximate stationary point. 
For sufficiently small $\epsilon$, this bound improves over gradient descent's bound. Moreover, this bound is the best possible guarantee on functions with Lipschitz Hessians \cite{carmon2017loweri}. 
They achieved this result using cubic regularized Newton’s method (CRN), which adds the cubic regularization term $\frac{L}{6} \| d \|^3$ to $M_k(d)$ in \eqref{eq:newton-direction}. This term plays a similar role to a trust-region, discouraging the search direction from being too \editTwo{large in norm}.

Unfortunately, CRN requires knowledge of the Lipschitz constant of the Hessian, which is rarely practical. To address this issue, \citet{cartis2011adaptiveI,cartis2011adaptiveII} developed adaptive cubic regularized Newton (ARC). ARC dynamically maintains a cubic regularization parameter $\sigma_k$ which can be interpreted as a local estimator of $L$. This estimator is decreased on steps with sufficient function value reduction and increased on steps with insufficient reduction.

Given that trust-region methods use second-order information and in practice converge in fewer iterations than gradient descent, one might anticipate they also obtain the $O(\Delta_f L^{1/2} \epsilon^{-3/2} )$ worst-case iteration bound achieved by CRN. 
Unfortunately, this is not true for the classical trust-region method of \citet{sorensen1982newton}.
For example, one can construct functions with $L$-Lipschitz Hessian and $\Delta_f < \infty$, such that classical trust-region methods require $\epsilon^{-2}$ iterations to find an $\epsilon$-approximate stationary point for any $\epsilon \in (0,1)$ \cite{newtons_method_slow_convergence_global_lip}.

Significant efforts have been made to develop an adaptive trust-region method matching the cubic regularization complexity bound \cite{jiang2023universal,curtis2017trust,curtis2021trust,curtis2023worst}, though the resulting methods have diverged substantially from the original method. For instance, TRACE \cite{curtis2017trust} achieves a complexity bound proportional to $\epsilon^{-3/2}$, but its approach differs in several key aspects from the original trust-region framework.
TRACE maintains a cubic regularization parameter (similar to the cubic regularization parameter in ARC) that the algorithm updates along with the trust-region radius \cite[Algorithm 1]{curtis2017trust}. This makes TRACE significantly more complex compared to the classical trust-region framework.
While \citet{jiang2023universal} and \citet[Algorithm 1]{curtis2021trust} are variants of the classical trust-region method, they do introduce a regularization parameter to 
the trust-region subproblem, which can degrade practical performance  \citet[Figure 1]{curtis2021trust}.

Finally, an important topic of study is developing methods that match  the whole $O(\Delta_f L^{1/2} \epsilon^{-3/2})$ bound, including the Lipschitz constant, not just the $\epsilon^{-3/2}$ scaling. 
From Table~\ref{table:adaptive-second-order-methods}, we can see that this has been achieved 
for adaptive cubic regularization methods \cite{grapiglia2017regularized,cartis2019universal}.
However, inspection of Table~\ref{table:adaptive-second-order-methods} also shows that for most existing trust-region methods, the $L$ scaling is greater than or equal to $L^{3/2}$, which is a factor of $L$ worse than the optimal scaling of $L^{1/2}$.

\editTwo{The only trust-region method that achieves a convergence bound close to the full optimal complexity bound is the Universal trust-region (UTR) method \cite{jiang2023universal} (which is a follow-up to the conference version of this paper \cite{hamad2022consistently}).
More precisely, their first-order complexity bound is 
\[
O\left( \Delta_f L^{1/2} \epsilon^{-3/2} \log\left( \frac{G_a}{\epsilon} \right) \log\left( \frac{L^{1/2}}{\rho_{\min}} + 1 \right) \right) + \tilde{O}(1)
\]
where $G_a$ is an upper bound on the gradients throughout the trajectory and $\rho_{\min} \le L^{1/2}$ is some user-specified lower bound on 
the square root of the Lipschitz constant.
Their bound is slightly looser:
while our bound matches the optimal rate up to additive logarithmic terms, their bound only matches up to multiplicative logarithmic factors.

The design of their method also deviates from our method in several meaningful ways. 
Firstly, they maintain an estimate 
of the square root of the Lipschitz constant of the Hessian, $\rho^{(k)}$, and update
the trust-region radius to be proportional to $\rho^{(k)} \| \grad f(x_k) \|^{1/2}$
and add a regularization parameter proportional to $\rho^{(k)} \| \grad f(x_k) \|^{1/2}$ which is 
disabled and enabled based on the 
value of $\lambda_{\min}(\grad^2 f(x_k))$. 
For a direct implementation this adds 1-2 additional factorizations per iteration.
For an indirect implementation this is impractical:
$(i)$ each iteration would require an iterative method to compute the minimum eigenvalue 
significantly increasing the cost of each iteration and $(ii)$ one would need a method that produces a 
certificate that the minimum eigenvalue is correctly calculated.
Secondly, they completely forgo  using predicted model reduction to decide whether to accept a step,
defining a successful step where they increase $\rho^{(k)}$ when either the function is reduced by $\rho^{(k)} \| \grad f(x_k) \|^{3/2}$ or the norm of the gradient is reduced by a constant factor. 
Therefore, compared with our method, their approach deviates quite significantly from a classical trust-region method.
Finally, their theory assumes that they can exactly solve the trust-region subproblems, which prohibits a practical implementation that maintains their theoretical guarantees.
}

\begin{table}[!thb]
	\caption{Adaptive second-order methods along with their worst-case  bounds on the number of gradient, function, and Hessian evaluations. $\sigma_{s}\in (0,\infty)$ is the smallest regularization parameter used by ARC \citep{cartis2011adaptiveII}.
		$\sigma_0 \in (0,\infty)$ is the initial regularization parameter for cubic regularized methods. Table sourced from \citet{hamad2022consistently}.
	}
	\label{table:adaptive-second-order-methods}
	\centering
	{\renewcommand{\arraystretch}{1.3}
		\resizebox{\textwidth}{!}{
		\begin{tabular}{lll}
			\toprule
			\textbf{Algorithm} & \textbf{Type} & \textbf{Worst-case iteration bound} \\\midrule
			\small \textbf{ARC \citep{cartis2011adaptiveII}} & cubic regularized & $O(\Delta_f  L^{3/2} \sigma_{s}^{-1} \epsilon^{-3/2} + \Delta_f \sigma_{s}^{1/2} \epsilon^{-3/2})$ \\
			\small \textbf{Nesterov et al. \citep[Eq. 5.13 and 5.14]{grapiglia2017regularized}}  & cubic regularized & $O(\Delta_f  \max\{L, \sigma_0 \}^{1/2} \epsilon ^ {-3/2}) + \tilde{O}(1)$ \\
			\small \textbf{ARp \citep[Section 4.1]{cartis2019universal}}  & cubic regularized &  $O(\Delta_f \max\{L, \sigma_0 \}^{1/2} \epsilon^{-3/2}) + \tilde{O}(1)$\\
			\small  \textbf{TRACE \citep[Section 3.2]{curtis2017trust}} & trust-region & $O(\Delta_f L^{3/2} \epsilon^{-3/2}) + \tilde{O}(1)$\\
			\small  \textbf{I-TRACE \citep[Section 3.2]{curtis2023worst}} & trust-region & $O(\Delta_f L^{3/2} \epsilon^{-3/2}) + \tilde{O}(1)$\\
			\small \textbf{Curtis et al. \citep[Section 2.2]{curtis2021trust}} & trust-region & $\tilde{O}\left( \Delta_f \max{\left\{ L^2, 1 + L \right\}} \epsilon^{-3/2} \right)$\\
			\small \textbf{Jiang et al. \citep[Section 3.1]{jiang2023universal}} & trust-region & $\tilde{O}\left( \Delta_f L ^ {1/2} \epsilon^{-3/2} \right)$\\
			\small \textbf{Newton-MR \citep[Section 3.2.1]{liu2022newton}} & line-search &  $O( \Delta_f \max\{L ^ 2, L^1\}  \epsilon^{-3/2})$\\
			\small \textbf{\citet[Section 4]{he2023newton}} & line-search &  $O( \Delta_f L ^{1/2} \epsilon^{-3/2})$\\
			\small \textbf{AN2C \citep[Section 3]{gratton2024yet}} & regularized Newton & $\tilde{O}( \Delta_f \max\{L ^ {1/2}, L^2\}  \epsilon^{-3/2}) +  \tilde{O}(1)$\\
			\small \textbf{Theorem~\ref{thm:main-fully-adaptive-result}} & trust-region & $O( \Delta_f L^{1/2}  \epsilon^{-3/2}) + \tilde{O}(1)$\\
			\bottomrule
		\end{tabular}
	}
	}
\end{table}

\paragraph{Outline} The paper is structured as follows. In \Cref{sec:our-tr-method}, we introduce our trust-region method. In \Cref{sec:main-result}, we present our main result: the optimal first-order iteration complexity of \Cref{alg:CAT}
\edit{ on functions with Lipschitz continuous Hessian. In \Cref{sec:locally-lipschitz-second-derivatives}, \editTwo{we provide 
convergence guarantees for our algorithm on continuously twice-differentiable} functions. In \Cref{sec:numerical-results}, we evaluate the performance of our algorithm on the CUTEst test set \citep{gould2015cutest} and discuss the experimental results.  In \Cref{sec:comparison-conference-version}, we compare the method developed in this paper with the conference version of this work \citep{hamad2022consistently}. Finally, \Cref{sec:Solving the trust-region sub-problem} describes our trust-region subproblem solver's implementation. }

\paragraph{Notation} Let  $\N$ be the set of natural numbers (starting from one), $\eye$ be the identity matrix, and $\R$ be the set of real numbers. Let $\|\cdot\|$ denote the $\ell_2$ norm for a vector and the spectral norm for a matrix.
Throughout this paper, we assume that $n \in \N$ and $f : \R^{n} \rightarrow \R$ is twice differentiable. For the remainder of this paper, $x_k$ and $d_k$ denote the iterates and search directions, respectively, of Algorithm~\ref{alg:CAT}.

\section{Our trust-region method}\label{sec:our-tr-method}
%
%

\newcommand{\AcceptableConst}{\sigma}

\begin{algorithm}
	\caption{\textbf{C}onsistently  \textbf{A}daptive \textbf{T}rust Region Method (CAT)}
	\label{alg:CAT}
	\begin{algorithmic}
		\STATE \textbf{Input requirements:} $r_{1} \in (0,\infty)$, $x_1 \in \R^{n}$
		\STATE \textbf{Problem-independent parameter requirements:} $\theta \in (0,1)$,  $\beta \in (0,1)$,  \possibleImprovement{$\AcceptableConst \in [0,\beta]$}, $\omega_1 \in (1,\infty), \omega_2 \in [\omega_1,\infty)$, $ 0 \leq \SPTone < \frac{1}{2} (1 - \frac{\beta \theta}{\SPTthree(1 - \beta)})$, \RequireSPT
		\FOR{$k = 1, \dots, \infty$}
		\STATE Approximately solve the trust-region subproblem, i.e., obtain a direction $d_k$ that satisfies \eqref{eq:subproblem-termination-criteria}
		\STATE $x_{k + 1} \gets 
		\begin{cases}
			x_k + d_k  \quad\quad\quad \text{if } f(x_k + d_k) \le f(x_k) \possibleImprovement{~\&~ \hat{\rho}_k \ge \AcceptableConst}  & \quad \text{(step is accepted)} \\
			x_k  \quad\quad\quad\quad\quad~ \text{otherwise}  & \quad \text{(step is rejected)} 
		\end{cases}$
		\STATE $r_{k + 1} \gets 
		\begin{cases}
			\max\{\omega_2 \| d_k \|, r_k \} \quad\quad &\text{if }  \hat{\rho}_k \ge \beta  \quad\quad\quad\quad~ \text{(step is successful)} \\
			r_k / \omega_1 \quad\quad\quad\quad\quad\quad\quad   &\text{if } \hat{\rho}_k <\beta   \quad\quad\quad\quad~ \text{(step is unsuccessful)}
		\end{cases}$
		\ENDFOR
	\end{algorithmic}
\end{algorithm}

\subsection{Our approach}\label{sec:our-trust-region-method}
\newcommand{\Mk}[1][k]{M_{#1}}

The selection of the radius $r_k$ significantly impacts the performance of trust-region methods.
At each iteration of classical trust-region methods \cite{sorensen1982newton,conn2000trust} we calculate the ratio between the actual reduction, $f(x_k) - f(x_k + d_k)$, and the predicted reduction from the second-order Taylor series model:
\begin{flalign}\label{eq:classic-rho-k}
	\rho_k := \frac{f(x_k) - f(x_k + d_k)}{-\Mk(d_k)}.
\end{flalign}
The radius $r_k$ is increased when $\rho_k$ is close to one and reduced when $\rho_k$ is close to zero \citep{sorensen1982newton}. Unfortunately, there are examples of functions with Lipschitz continuous Hessians for which classical trust-region methods exhibit a convergence rate proportional to $\epsilon^{-2}$  \citep[Section~3]{cartis2010complexity} instead of $\epsilon^{-3/2}$. 

To address this issue, we modify \eqref{eq:classic-rho-k} by adding a term that depends on the norms of both the search direction and the gradient: $\frac{\theta}{2} \min\{\|\grad f(x_k)\|, \|\grad f(x_k + d_k) \| \} \| d_k \|$ to the predicted reduction, yielding a new ratio:
\begin{equation}\label{eq:rho-hat-k}
	\hat{\rho}_k := \frac{f(x_k) - f(x_k + d_k)}{-\Mk(d_k) + \frac{\theta}{2} \min\{\|\grad f(x_k)\|, \|\grad f(x_k + d_k) \| \} \| d_k \| } 
\end{equation}
where $\theta \in (0,1)$ is a problem-independent hyperparameter (we use $\theta = 0.1$ in our implementation).
The idea behind \eqref{eq:rho-hat-k} is that on successful steps, i.e., when $\hat{\rho}_k \ge \beta$, we have 
$f(x_k) - f(x_k + d_k) \ge \frac{\beta \theta}{2} \min\{\|\grad f(x_k)\|, \|\grad f(x_k + d_k) \| \} \| d_k \|$.
Moreover, if $\min\{\|\grad f(x_k)\|, \|\grad f(x_k + d_k) \| \} \ge \epsilon$, then $f(x_k) - f(x_k + d_k) \ge \frac{\beta \theta}{2} \epsilon \| d_k \|$.
Thus, on successful steps with a large search direction, we can guarantee significant reduction in the function value.

Our complete trust-region method is presented in Algorithm~\ref{alg:CAT}. 
Aside from replacing $\rho_k$ with $\hat{\rho}_k$, other differences from classical trust-region methods \citep{sorensen1982newton} include: (i) \editTwo{when $\sigma=0$, as in the numerical experiments, every direction 
that reduces the objective is accepted}, in contrast to standard trust-region methods \citep{sorensen1982newton}, which require the ratio $\rho_k$ to be above a certain threshold $\eta \in (0,1)$, and (ii) when $\hat{\rho}_k \geq \beta$, we update $r_{k+1}$ to $\max\{\omega_2 \|d_k\|, r_k\}$ instead of setting $r_{k+1}$ equal to $\omega_2 r_k$ as in classical trust-region methods \citep{sorensen1982newton}.

\subsection{Trust-region subproblem termination criteria}\label{sec:our-trust-region-subproblem}

Finding a solution to \eqref{eq:trust-region-subproblem} is often nontrivial, primarily because the Hessian of the objective function may not be positive definite ($\grad ^ 2 f(x_k) \nsucceq 0$) or because the unconstrained minimizer lies outside the trust-region ($\| \grad ^ 2 f(x_k) ^ {-1} \grad f(x_k)\| > r_k$) \citep{nocedal2006numerical}. Therefore, validating that a given search direction $d_k$ satisfies \eqref{eq:trust-region-subproblem} is a crucial practical consideration.

Fortunately, an exact solution for the trust-region subproblem \eqref{eq:trust-region-subproblem} is given by the following well-known fact.

\begin{fact}[Theorem 4.1 \citep{nocedal2006numerical}]   \label{fact:optimality-conditions-quadratic-model}
	The direction $d_k$ exactly solves \eqref{eq:trust-region-subproblem} if and only if there exists $\delta_k \in [0,\infty)$ such that:
	\begin{subequations}\label{eq:optimality-conditions-tr}
		\begin{flalign}
			\label{eq:model-gradient1}\grad \Mk(d_k) + \delta_k d_k &= 0 \\
			\delta_k r_k &\le \delta_k \| d_k \|  \\
			\| d_k \| &\le r_k \\
			\grad^2 f(x_k) + \delta_k \eye &\succeq 0  \label{eq:PSD-grad-squared-f-delta}
		\end{flalign}
	\end{subequations}
	which solves \eqref{eq:trust-region-subproblem}.
\end{fact}
Exactly solving the trust-region subproblem defined in \eqref{eq:trust-region-subproblem} is not always possible. Instead, it suffices to solve the subproblem approximately. Accordingly, our approach seeks a search direction $d_k$ and a regularization parameter $\delta_k \editTwo{\in [0,\infty)}$ that satisfy the following system of equations:
\begin{subequations}
\label{eq:subproblem-termination-criteria}
	\begin{flalign}
		\label{eq:model-gradient-weaker}\| \grad \Mk(d_k) + \delta_k d_k \| &\le \SPTone \varepsilon_k \\
		\label{eq:comp-delta-radius1} \SPTtwo \delta_k r_k  &\le \delta_k \| d_k \| \\
		\| d_k \| &\le r_k  \\
		\label{eq:model-reduction-by-dist1} \Mk(d_k) &\le -\SPTthree \frac{\delta_k}{2} \| d_k \|^2.
	\end{flalign}
\end{subequations}
where $\SPTone$, $\SPTtwo$, and $\SPTthree$ are problem-independent scalars such that  $ 0 \leq \SPTone < \frac{1}{2} \big(1 - \frac{\beta \theta}{\SPTthree(1 - \beta)} \big)$, \RequireSPT, and  $\varepsilon_k$ is defined as
\begin{flalign}\label{def:vareps}
	\varepsilon_{k+1} := \begin{cases}  
		\| \grad f(x_1) \| & k = 0 \\
		\min\{ \varepsilon_{k},  \| \grad f(x_{k} + d_{k}) \| \} & f(x_{k} + d_{k}) \le f(x_{k}) + \GainThres_{k} \\
		\varepsilon_{k} & \text{otherwise.}
	\end{cases}
\end{flalign}

For our theory to hold, we assume $\GainThres_k$ is any sequence of positive numbers (chosen by the algorithm designer) such that there exists some problem-independent constant $\AlgConst \in (0,\infty)$ such that
\begin{flalign}\label{eq:gain-thres-lower-bound}
	\GainThres_k \ge \AlgConst \varepsilon_{k} \| d_{k} \|.
\end{flalign}
Note that $\GainThres_k = \infty$ is a valid choice, which corresponds to 
\[
\varepsilon_{k+1} = \begin{cases}  
	\| \grad f(x_1) \| & k = 0 \\
	\min\{ \varepsilon_{k},  \| \grad f(x_{k} + d_{k}) \| \} & \text{otherwise}.
\end{cases}
\]
This choice evaluates $\| \grad f(x_{k} + d_{k}) \|$ at every step, even if the step is not accepted, i.e., when $x_{k+1} = x_{k}$. This can be wasteful in terms of gradient evaluations. Choosing a smaller $\GainThres_k$ \editTwo{typically} reduces the number of gradient evaluations because we \editTwo{need not compute $\| \grad f(x_{k} + d_{k}) \|$ when $f(x_k+d_k)>f(x_k)+\GainThres_k$.}
For our practical implementation, we pick 
$\GainThres_k =  \defaultAlgConst \varepsilon_{k} \| d_{k} \| + 10^{-8} (\abs{f(x_k)} + 1 )$, where the $10^{-8} (\abs{f(x_k)} + 1 )$ term is designed to mitigate arithmetic errors in evaluating the difference $f(x_k+d_k) - f(x_k)$. This is important because arithmetic errors are most likely to occur immediately prior to termination, i.e., when $\| \grad f(x_k + d_k) \| \le \epsilon$.
\editTwo{The ablation study in \Cref{sec:ablation-study-results} shows that 
setting $\GainThres_k = 0$ decreases 
the shifted geometric mean number of gradient evaluations by only $3.5\%$ and has no impact on the wall-clock time.
}

It is straightforward (e.g., \cite[Lemma 1]{hamad2022consistently}) to demonstrate that solving the trust-region subproblem exactly provides a solution to the system \eqref{eq:subproblem-termination-criteria}, i.e., with $\SPTone=0$, $\SPTtwo=1$, and $\SPTthree=1$.
 However, the converse is not true; an exact solution to \eqref{eq:subproblem-termination-criteria} does not guarantee a solution to the trust-region subproblem. Nonetheless, these conditions suffice for  our results and 
are more computationally tractable to verify than \eqref{eq:optimality-conditions-tr}  which due to \eqref{eq:PSD-grad-squared-f-delta} requires a computationally expensive eigenvalue calculation.

\editTwo{The subproblem termination criteria given in \Cref{eq:subproblem-termination-criteria} are 
specifically designed to obtain the $O(\Delta_f L^{1/2} \epsilon^{-3/2}) + \widetilde{O}(1)$ convergence bound.
In particular, the primary role of 
\Cref{eq:model-gradient-weaker}, \Cref{eq:comp-delta-radius1} and \Cref{eq:model-reduction-by-dist1} is to establish 
\Cref{lem:gradient-bound_distance} which provides an upper bound on the gradient at 
the next iterate. Similarly, the primary role of \Cref{eq:comp-delta-radius1} is to establish 
\Cref{lem:radius-and-direction-bounds}.\ref{lem:radius-and-direction-bounds:part-general-radius-bounds-part2} to maintain an upper bound on the radius in terms of the number of 
unsuccessful and successful steps.
}

%

\section{Global convergence bound for our method on \edit{functions with Lipschitz continuous Hessian}}\label{sec:main-result}

In this section, we provide convergence guarantees for our method in terms of iteration complexity for finding approximate stationary points on functions with $L$-Lipschitz Hessians. In particular, we prove that our method finds an $\epsilon$-approximate stationary point in at most $O(\Delta_f  L^{1/2} \epsilon^{-3/2}) + \tilde{O}(1)$ iterations. 
This analysis relies on the assumption that the iterates and search directions satisfy 
\Cref{assume:nice-directions}. 
\Cref{eq:bounded-f-sequence} immediately holds if the function is bounded below.
It is well-known that if the Hessian of $\grad^2 f$ is $L$-Lipschitz then \Cref{eq:Lip-grad1} and \eqref{eq:Lip-progress1} holds \citep[Lemma 1]{nesterov2006cubic}.  		
		
\begin{assumption}\label{assume:nice-directions}
Suppose that for all $k \in \N$ that
	\begin{flalign}
		f(x_1) - f(x_k)&\le \Delta_f \label{eq:bounded-f-sequence} \\
		\| \grad f(x_k+d_k) \| &\le \| \grad \Mk(d_k) \| + \frac{L }{2} \| d_k \|^2  \label{eq:Lip-grad1}\\
		f(x_k + d_k) &\le f(x_k) + \Mk(d_k) + \frac{L}{6} \| d_k \|^3 \label{eq:Lip-progress1} 
	\end{flalign}
for some $L \in (0,\infty)$ and $\Delta_f \in (0,\infty)$.
\end{assumption}

\newcommand{\AlgorithmConstantCOne}{c_1} 
\newcommand{\AlgorithmConstantCTwo}{c_2}
We will use the following problem-independent constants throughout our proofs:
\edit{
	\begin{flalign}\label{eq:algorithm-constants-definition}
	\AlgorithmConstantCOne := \frac{2 + 3 \SPTthree (1 - \beta)}{6 (\SPTthree(1 - \beta) - \beta \theta)} \And \AlgorithmConstantCTwo := \frac{\SPTthree(1-\beta)}{\SPTthree(1-\beta) - \beta \theta}.
	\end{flalign}
}
Note that $\AlgorithmConstantCOne > \frac{1}{2}$ because $\AlgorithmConstantCOne = \frac{2 + 3 \SPTthree (1 - \beta)}{6 (\SPTthree(1 - \beta) - \beta \theta)} > \frac{3 \SPTthree (1 - \beta)}{6 \SPTthree(1 - \beta)} = \frac{1}{2}$. 
Additionally, $\AlgorithmConstantCTwo > 1$ because according to the requirements of Algorithm~\ref{alg:CAT}, we have $1 - \frac{\beta \theta}{\SPTthree(1 - \beta)} > 0$, which implies $\SPTthree (1 - \beta) > \SPTthree (1 - \beta) - \beta \theta > 0$. We will find these two facts useful in the proof of Lemma~\ref{lem:gradient-bound_distance}.

\editTwo{Lemma~\ref{lem:gradient-bound_distance} shows that, under specific conditions,
the norm of the gradient for the candidate solution $x_k + d_k$
can be controlled by two quantities: the size of the search direction and the residual in the stationarity 
conditions for the trust-region subproblem (see \Cref{eq:model-gradient-weaker}). 
}

\begin{lemma} \label{lem:gradient-bound_distance}
	Suppose \edit{\Cref{assume:nice-directions} holds}. If
	$\| d_k \| < \SPTtwo r_k$ or \edit{$\hat{\rho}_k < \beta$} then
	\begin{flalign}\label{eq:grad-bound-distance1}
		\| \grad f(x_k + d_k) \| \le \AlgorithmConstantCOne L \|d_k\|^2 + 
		\AlgorithmConstantCTwo \| \grad \Mk (d_k) + \delta_k d_k \|.
	\end{flalign}
\end{lemma}

\begin{proof}
	\possibleImprovement{First, consider the case that 
	$\| d_k \| = 0$.
	In this case, \Cref{eq:model-gradient-weaker} implies that 
	$\| \grad f(x_k) \| =\| \grad \Mk (d_k) \| \le \SPTone \varepsilon_k \le \SPTone \| \grad f(x_k) \|$ which implies $\| \grad f(x_k) \| = 0$ and since $\| \grad f(x_k + d_k) \| = \| \grad f(x_k) \|$ the lemma holds.
	Thus, for the remainder of the proof we assume $\| d_k \| \neq 0$.
	
	Next,} we prove the result in the case $\|d_k\| < \SPTtwo r_k$. By \eqref{eq:comp-delta-radius1} the statement $\| d_k \| < \SPTtwo r_k$ implies $\delta_k = 0$. Thus,
	\[
	\| \grad f(x_k + d_k) \| \underle{(a)} \| \grad \Mk (d_k) + \delta_k d_k \| + \frac{L}{2} \|d_k\| ^ 2 \underle{(b)} \AlgorithmConstantCOne L \|d_k\| ^ 2 + \AlgorithmConstantCTwo \| \grad \Mk (d_k) + \delta_k d_k \|
	\]
	where inequality $(a)$ uses \eqref{eq:Lip-grad1} with $\delta_k = 0$, and inequality $(b)$ uses that $\AlgorithmConstantCOne > 1/2$ and $\AlgorithmConstantCTwo > 1$ as per the discussion following their definitions. Next, we prove the result for the case that \edit{$ \hat{\rho}_k < \beta$.} Observe that
	\begin{flalign*}
		&\beta \left( -\Mk(d_k) + \frac{\theta}{2}\min\{\|\grad f(x_k)\|, \|\grad f(x_k + d_k) \| \} \| d_k \|\right) \\
		&\edit{\undergreater{(a)}} \hat{\rho}_k  \left( -\Mk(d_k) + \frac{\theta}{2} \min\{\|\grad f(x_k)\|, \|\grad f(x_k + d_k) \| \} \| d_k \| \right) \\
		&\undereq{(b)} f(x_k) - f(x_k + d_k) \underge{(c)} -\Mk(d_k) - \frac{L}{6} \| d_k \|^3
	\end{flalign*}
	where inequality (a) \edit{uses $-\Mk(d_k) + \frac{\theta}{2} \min\{\|\grad f(x_k)\|, \|\grad f(x_k + d_k) \| \} \| d_k \| \ge 0$ and $\hat{\rho}_k < \beta$}, equality (b) uses the definition of $\hat{\rho}_k$ in \eqref{eq:rho-hat-k}, and inequality (c) uses \eqref{eq:Lip-progress1}.
	Rearranging the previous displayed inequality and then applying \eqref{eq:model-reduction-by-dist1} using $1 - \beta > 0$ yields:
	\begin{flalign*}
		\frac{L}{6} \| d_k \|^3 + \frac{\beta \theta}{2}\min\{\|\grad f(x_k)\|, \|\grad f(x_k + d_k) \| \} \|d_k\|
		&\ge -(1-\beta) \Mk(d_k) \\
		& \ge (1-\beta) \SPTthree \frac{ \delta_k}{2} \| d_k \| ^ 2.
	\end{flalign*}
	Dividing this inequality by $\frac{(1-\beta) \|d_k\|}{2}$, using $1-\beta > 0$, and then using $ \|\grad f(x_k + d_k) \| \geq \min\{\|\grad f(x_k)\|, \|\grad f(x_k + d_k) \| \}$, we get 	\begin{flalign}\label{eq:lemma13:0328j}
		\frac{L}{3 (1 - \beta)} \| d_k \|^2 + \frac{\beta \theta}{1 - \beta} \| \grad f(x_k + d_k) \|
		\ge \SPTthree \delta_k \| d_k \|.
	\end{flalign}	 
	Now, by \eqref{eq:Lip-grad1}, the triangle inequality, and \eqref{eq:lemma13:0328j} respectively:
	\begin{flalign*}
		&\| \grad f(x_k + d_k) \| \le \| \grad \Mk(d_k) \| + \frac{L}{2} \| d_k \|^2  \le \| \grad \Mk(d_k) + \delta_k d_k \| + \delta_k \| d_k \| + \frac{L}{2} \| d_k \|^2 \\
		&\le \| \grad \Mk(d_k) + \delta_k d_k \|  + L \left( \frac{1}{3 \SPTthree(1 - \beta)} + \frac{1}{2} \right) \| d_k \|^2 + \frac{\beta \theta}{\SPTthree(1 - \beta)} \| \grad f(x_k + d_k) \|.
	\end{flalign*} Rearranging this inequality for $\| \grad f(x_k + d_k) \|$ and using $\frac{\beta \theta}{\SPTthree(1 - \beta)}  < \frac{\beta \theta}{\SPTthree(1 - \beta)} + 2 \SPTone < 1$ from the requirements of Algorithm~\ref{alg:CAT} yields:
	\begin{flalign*}
		\| \grad f(x_k + d_k) \|
		&\le \frac{1}{1 - \frac{\beta \theta}{\SPTthree(1 - \beta)}} \| \grad \Mk(d_k) + \delta_k d_k \| + \frac{\frac{1}{3 \SPTthree (1 - \beta)} + \frac{1}{2}}{1 - \frac{\beta \theta}{\SPTthree(1 - \beta)}} L \| d_k \|^2 \\
		&= \frac{\SPTthree(1 - \beta)}{\SPTthree(1 - \beta) - \beta \theta} \| \grad \Mk(d_k) + \delta_k d_k \| + \frac{2 + 3 \SPTthree (1 - \beta)}{6 (\SPTthree (1 - \beta) - \beta \theta)} L\| d_k \|^2 \\
		&= \AlgorithmConstantCOne L \|d_k\| ^ 2 + \AlgorithmConstantCTwo \| \grad \Mk(d_k) + \delta_k d_k \|.
	\end{flalign*}
\end{proof}
For the remainder of this section, we will find the following quantities useful:
\begin{align}
	\label{eq:constant-algorithm-definition}
	\TheoryAlgConst &:= \max\left\{\frac{1}{3\AlgConst}, \frac{2 \AlgorithmConstantCOne}{1 - 2 \SPTone \AlgorithmConstantCTwo} \right\} \\
	\dminInd{k} &:= \SPTtwo \omega_1 ^ {-1}\TheoryAlgConst^{-1/2} L^{-1/2} \varepsilon_k^{1/2}  \label{d-min-def}\\
	\dmaxInd{k} &:= \frac{2 \omega_2}{\beta \theta} \cdot \frac{\Delta_f}{\varepsilon_k} \label{d-max-def}
\end{align}
\edit{where $\AlgorithmConstantCOne$ and $\AlgorithmConstantCTwo$ are problem-independent constants defined in~\eqref{eq:algorithm-constants-definition} and $\AlgConst$ is also a problem-independent constant defined in~\eqref{eq:gain-thres-lower-bound}.}

\edit{By the definition of $\varepsilon_k$, we have that $\varepsilon_{k+1} \le \varepsilon_k$ for all $k \in \N$, which implies that the sequence $\{\dminInd{k}\}$ is monotonically decreasing, while the sequence $\{\dmaxInd{k}\}$ is monotonically increasing. These monotonicity properties will be used later in the proof.
}

Lemma~\ref{lem:radius-and-direction-bounds} translates Lemma~\ref{lem:gradient-bound_distance} into explicit bounds on the trust-region radius and search direction sizes. 

\begin{lemma}\label{lem:radius-and-direction-bounds}
	Suppose $k \in \N$, and $j \in [k,\infty) \cap \N$, then
	\begin{enumerate}[label=(\roman*)]
		\item If \editTwo{\Cref{eq:bounded-f-sequence}} holds then $\omega_1^{k - j} r_k \le r_j \le \max\{r_k, \dmaxInd{j+1}\}$. \label{lem:radius-and-direction-bounds:part-general-radius-bounds-part1}
		\item We have $r_j \leq r_k \omega_1 ^ {-n_j} \omega_2 ^ {p_j}$ where $p_j := \abs{ \{ m \in [k, j) : \hat{\rho}_m \geq \beta \}}$ and $n_j := \abs{\{m \in [k, j) : \hat{\rho}_m < \beta \}}$.\label{lem:radius-and-direction-bounds:part-general-radius-bounds-part2}
		\item If \Cref{assume:nice-directions} holds and $\varepsilon_{j+1} \geq \frac{\varepsilon_k}{2}$ then $\min\{ r_k \SPTtwo(\omega_2 \SPTtwo)^{j-k}, \dminInd{k}  \} \le \| d_j \| $.
		\label{lem:radius-and-direction-bounds:large-gradient}
	\end{enumerate}
\end{lemma}

\begin{proof}[\editTwo{Proof of \Cref{lem:radius-and-direction-bounds}.\ref{lem:radius-and-direction-bounds:part-general-radius-bounds-part1}}]
	We will first show $\omega_1 ^ {k - j} r_k \le r_j$. Consider the induction hypothesis that if $j \geq k$ then $r_j \geq \omega_1 ^ {k - j} r_k$. The hypothesis holds for the base case when $j = k$. Next, suppose that the hypothesis holds for some $j = t$. Then \editTwo{by the radius update rule of Algorithm~\ref{alg:CAT} and the induction hypothesis we have}
	\editTwo{$r_{t + 1} \ge \frac{r_t}{\omega_1}  \ge \omega_1 ^ {k - (t + 1)} r_k$}. Therefore, the induction hypothesis holds for $j = t + 1$. By induction our claim holds. 
	
	Next, we show that $r_{j} \le \max\{r_k, \dmaxInd{j+1}\}$. To prove this result it will suffice to establish that $r_{j+1} \le \max\{r_j, \dmaxInd{j+1}\}$ and the result will follow by induction. If $\|d_j\| \le \dmaxInd{j+1} / \omega_2$  then using the update rule of Algorithm~\ref{alg:CAT}, \edit{we get that}
	$r_{j+1} \leq \max\{ \omega_2 \| d_j \|,  r_j \} \leq \max\{\dmaxInd{j+1}, r_j\}$. If $x_{j+1} = x_j$, then by inspection of Algorithm \ref{alg:CAT} we have $r_{j+1} = r_j / \omega_1 \le \max\{\dmaxInd{j+1}, r_j\}$.
	Thus the only remaining case to consider is $x_{j+1} = x_j + d_j$ and $\|d_j\| > \dmaxInd{j+1} / \omega_2$.
	In this case we have
	\begin{flalign*}
		\hat{\rho}_j &= \frac{f(x_j) - f(x_j + d_j)}{-\Mk[\editTwo{j}](d_j) + \frac{\theta}{2} \min\{\|\grad f(x_j)\|, \|\grad f(x_j + d_j) \| \} \| d_j \|} \\ 
		&\underle{(a)} \frac{f(x_j) - f(x_{j+1})}{\frac{\theta}{2} \min\{\|\grad f(x_j)\|, \|\grad f(x_j + d_j) \| \} \| d_j \|}\\ 
		&\underle{(b)} 2\frac{f(x_j) - f(x_{j+1})}{\theta \varepsilon_{j+1} \| d_j \|} \underle{(c)} 2\frac{\Delta_f}{ \theta\varepsilon_{j+1} \| d_j \|} 
		< 2\frac{\Delta_f \omega_2}{ \theta\varepsilon_{j+1} \dmaxInd{j+1}} \undereq{(d)} \beta
	\end{flalign*}
	where inequality $(a)$ follows from $-\Mk[\editTwo{j}](d_j) \ge 0$ and \editTwo{$f(x_{j+1}) = f(x_j + d_j)$}, inequality $(b)$ follows from the fact that $\min\{\|\grad f(x_j)\|, \|\grad f(x_j + d_j) \| \} \ge \varepsilon_{j+1}$ \edit{due to~\eqref{def:vareps}}, inequality $(c)$ uses \edit{\eqref{eq:bounded-f-sequence}} and $f(x_1) \ge f(x_j)$, and
	equality $(d)$ uses the definition of $\dmaxInd{j+1}$.
	Thus, $\hat{\rho}_j < \beta$, and by inspection of Algorithm \ref{alg:CAT} we have $r_{j+1} = r_j  / \omega_1 \leq \max\{r_j,  \dmaxInd{j+1}\}$ as desired.
\end{proof}

\begin{proof}[\editTwo{Proof of \Cref{lem:radius-and-direction-bounds}.\ref{lem:radius-and-direction-bounds:part-general-radius-bounds-part2}}]	
Consider the induction hypothesis that if $j \geq k$ then $r_j \le r_k \omega_1^{-n_j} \omega_2^{p_j}$. The hypothesis holds for the base case when $j = k$, because $p_j = n_j = 0$. Next, suppose that the hypothesis holds for some $j = t$. If $\hat{\rho}_{t} \geq \beta$, then \editTwo{by the radius update rule of Algorithm~\ref{alg:CAT}, $\omega_2 > 1$ and $\|d_t\| \leq r_t$ we get
$r_{t+1} \le \omega_2 r_t$; 
thus $r_{t+1} \le \omega_2 r_t \le r_k \omega_1^{-n_{t}} \omega_2^{p_{t}+1} = r_k \omega_1^{-n_{t+1}} \omega_2^{p_{t+1}}$ as desired.} On the other hand, if $\hat{\rho}_{t} < \beta$, then
\editTwo{by the radius update rule of Algorithm~\ref{alg:CAT} we have $r_{t+1} = r_t / \omega_1$;
thus $r_{t+1} \le r_t / \omega_1 \le r_k\omega_1^{-n_{t}-1} \omega_2^{p_{t}} = r_k\omega_1^{-n_{t+1}} \omega_2^{p_{t+1}}$.} By induction, our claim holds.
\end{proof}

\editTwo{To prove \Cref{lem:radius-and-direction-bounds}.\ref{lem:radius-and-direction-bounds:large-gradient}
we first prove \Cref{claim:generic-direction-lower-bound};
the proof of \Cref{lem:radius-and-direction-bounds}.\ref{lem:radius-and-direction-bounds:large-gradient} 
then proceeds to establish the premise of 
\Cref{claim:generic-direction-lower-bound}.
We will also find \Cref{claim:generic-direction-lower-bound}
useful later in \Cref{sec:lipschitz-first-derivatives}.}

\newcommand{\GenericDmin}{\mathfrak{\ubar{d}}}

\editTwo{
\begin{claim}\label{claim:generic-direction-lower-bound}
Suppose that for all $j \in [k,K]$, there exists some threshold $\GenericDmin \in (0,\infty)$ such that 
\begin{flalign}\label{eq:generic-direction-lower-bound-premise}
\text{if $\|d_j\| < \SPTtwo r_j$ or $r_{j+1} < \omega_2 \|d_j\|$ then 
$\| d_j \| \geq \omega_1 \SPTtwo ^{-1} \GenericDmin$}.
\end{flalign}
Then, for all $j \in [k,K]$, 
\begin{flalign}\label{eq:generic-direction-lower-bound}
\min\{ r_k \SPTtwo(\omega_2 \SPTtwo)^{j-k}, \GenericDmin \} \le \| d_j \|.
\end{flalign}
\end{claim}
}

\begin{proof}
	\editTwo{We will establish \Cref{eq:generic-direction-lower-bound}} by induction.
	First, consider the base case when $j = k$.
	If $\SPTtwo r_k \le \| d_k \|$ then $\min\{ r_k \SPTtwo(\omega_2 \SPTtwo)^{j-k}, \editTwo{\GenericDmin}  \} = \min\{ r_k \SPTtwo, \editTwo{\GenericDmin}  \}  \le \SPTtwo r_k \le \editTwo{\| d_k \| =} \| d_j \|$.
	On the other hand, if $\SPTtwo r_k > \| d_k \|$ then
	\[
	\min\{ r_k \SPTtwo(\omega_2 \SPTtwo)^{j-k}, \editTwo{\GenericDmin}  \} = \min\{ r_k \SPTtwo, \editTwo{\GenericDmin}  \}  \le \editTwo{\GenericDmin} \underle{(a)} \omega_1 \SPTtwo ^{-1} \editTwo{\GenericDmin} \underle{(b)} \editTwo{\| d_k \| =} \| d_j \|
	\] 
	where inequality $(a)$ uses that $\omega_1 \SPTtwo ^{-1} \ge 1$ from the requirements of Algorithm~\ref{alg:CAT} and inequality $(b)$ uses \eqref{eq:generic-direction-lower-bound-premise} with $j=k$. This establishes the base case.
	
	Next, suppose that the induction hypothesis $\min\{ r_k \SPTtwo(\omega_2 \SPTtwo)^{j-k}, \editTwo{\GenericDmin}  \} \le \| d_j \|$ holds for $j=t-1$ and $\varepsilon_{\editTwo{j}+1} \ge \frac{\varepsilon_{k}}{2}$.
	We will split the proof of showing that the induction hypothesis holds for $j=t$ into three different cases.
	In the
	case that $\| d_{t} \| < \SPTtwo r_{t}$ then by \eqref{eq:generic-direction-lower-bound-premise} with $j = t$ we get $\| d_{t} \| \ge \omega_1 \SPTtwo ^{-1} \editTwo{\GenericDmin} \ge \editTwo{\GenericDmin}$ as desired.
	In the case that $\| d_t \| \ge \SPTtwo r_t$ and $r_{t} < \omega_2 \| d_{t-1} \|$ we have 
	\[ 
	\| d_t \| \ge \SPTtwo r_{t} \underge{(a)} \SPTtwo \omega_1^{-1} r_{t-1} \ge \SPTtwo \omega_1^{-1} \| d_{t-1} \| \underge{(b)} \editTwo{\GenericDmin}
	\]
	where inequality $(a)$ is from the update rule for Algorithm~\ref{alg:CAT} which implies $r_t \ge \frac{r_{t-1}}{\omega_1}$ and inequality $(b)$ is from \eqref{eq:generic-direction-lower-bound-premise} with $j = t-1$.
	Finally, in the case that $\| d_t \| \ge \SPTtwo r_t$ and $r_{t} \ge \omega_2 \| d_{t-1} \|$ we have 
	\begin{flalign*}
		\| d_t \| \ge \SPTtwo r_t \ge \SPTtwo \omega_2 \| d_{t-1} \| &\underge{(a)} \SPTtwo \omega_2 \min\{ r_{k} \SPTtwo(\omega_2 \SPTtwo)^{t-1-k}, \editTwo{\GenericDmin}  \}\\ &\underge{(b)} \min\{ r_{k} \SPTtwo(\omega_2 \SPTtwo)^{t-k}, \editTwo{\GenericDmin}  \}
	\end{flalign*} 
	where inequality $(a)$ uses the induction hypothesis and inequality $(b)$ uses $\SPTtwo \omega_2 \ge \SPTtwo \omega_1 > 1$ by the requirements of Algorithm~\ref{alg:CAT}.   
	Thus $\min\{ r_k \SPTtwo(\omega_2 \SPTtwo)^{j-k}, \editTwo{\GenericDmin}  \} \le \| d_j \|$ holds for $j=t$ and by induction \editTwo{the result holds.}
\end{proof} 

\begin{proof}[\editTwo{Proof of \Cref{lem:radius-and-direction-bounds}.\ref{lem:radius-and-direction-bounds:large-gradient}}]
	To prove this result, \editTwo{it suffices to establish the premise of \Cref{claim:generic-direction-lower-bound}
	with $\GenericDmin = \dminInd{k}$.}
	First, consider the case that $\| \grad f(x_j + d_j)\| < \varepsilon_{j+1}$. Then,
	\[
	\frac{\TheoryAlgConst^{-1} \varepsilon_k \| d_j \|}{6} \underle{(a)} 
	\frac{\AlgConst\varepsilon_k}{2} \|d_j\| \underle{(b)} 
	\AlgConst \varepsilon_{j + 1} \|d_j\| \le \AlgConst \varepsilon_{j} \| d_{j} \| \underle{(c)} 
	\GainThres_j \underless{(d)}  f(x_j + d_j) - f(x_j) \underle{(e)} 
	\frac{L \| d_j \|^3}{6}
	\]
	where inequality $(a)$ uses \eqref{eq:constant-algorithm-definition}, inequality (b) uses the assumption that $\varepsilon_{j+1} \geq \frac{\varepsilon_k}{2}$, inequality (c) uses \eqref{eq:gain-thres-lower-bound},
	inequality $(d)$ uses $\| \grad f(x_j + d_j)\| < \varepsilon_{j+1}$ and the definition of $\varepsilon_{j+1}$ as given in \eqref{def:vareps}, and inequality (e) uses $M_j(d_j) \le 0$ and \eqref{eq:Lip-progress1}. Rearranging this inequality yields $\| d_j \| \ge \sqrt{ \varepsilon_k \TheoryAlgConst^{-1} L ^ {-1}}  = \omega_1 \SPTtwo ^{-1} \dminInd{k}$ as desired.
	
	Next, consider the case that $\| \grad f(x_j + d_j)\| \ge \varepsilon_{j+1}$.  If $\|d_j\| < \SPTtwo r_j$ then the premise of Lemma~\ref{lem:gradient-bound_distance} holds. On the other hand, if $r_{j+1} < \omega_2 \| d_j \|$ then by inspection of Algorithm~\ref{alg:CAT} we have $\hat{\rho}_j < \beta$. Therefore, either $\|d_j\| < \SPTtwo r_j$ or $r_{j+1} < \omega_2 \|d_j\|$ imply the premise of Lemma~\ref{lem:gradient-bound_distance} holds. 
	It follows that 
	\begin{flalign}\label{eq:varepsilon-k-dj-L}
		\frac{\varepsilon_k}{2} \underle{(a)} \varepsilon_{j+1} \underle{(b)} \|\grad f(x_j + d_j)\| \underle{(c)} \AlgorithmConstantCOne L \|d_j\| ^ 2 + \SPTone \AlgorithmConstantCTwo  \varepsilon_j \underle{(d)} \AlgorithmConstantCOne L \|d_j\| ^ 2 +  \SPTone \AlgorithmConstantCTwo \varepsilon_k
	\end{flalign}
	where inequality $(a)$ is by the assumption of the Lemma, inequality $(b)$ uses that we are analyzing the case that $\|\grad f(x_j + d_j)\| \ge \varepsilon_{j+1}$, inequality $(c)$ uses Lemma~\ref{lem:gradient-bound_distance} and \eqref{eq:model-gradient-weaker}, and inequality $(d)$ uses that $\varepsilon_j \le \varepsilon_k$ by definition of $\varepsilon_j$.
	
	According to the requirements of Algorithm~\ref{alg:CAT}: $0 \leq \SPTone < \frac{1}{2} (1 - \frac{\beta \theta}{\SPTthree (1 - \beta)}) =  \frac{\SPTthree (1 - \beta) - {\beta \theta}}{2 \SPTthree (1 - \beta)}$, it follows that
	\[ 
	\SPTone \AlgorithmConstantCTwo = \SPTone \cdot \frac{\SPTthree (1 - \beta)}{\SPTthree (1 - \beta) - \beta \theta} < \frac{\SPTthree (1 - \beta) - \beta \theta}{2 \SPTthree (1 - \beta)} \cdot \frac{\SPTthree (1 - \beta)}{\SPTthree (1 - \beta) - \beta \theta} = \frac{1}{2}.
	\]
	Thus, we can rearrange \eqref{eq:varepsilon-k-dj-L} to yield
	\begin{flalign*}
		\varepsilon_k \leq \frac{\AlgorithmConstantCOne L}{1/2 - \SPTone \AlgorithmConstantCTwo} \| d_j \|^2 = \frac{2 \AlgorithmConstantCOne}{1 - 2 \SPTone \AlgorithmConstantCTwo} L \| d_j \|^2 \le \TheoryAlgConst L \| d_j \|^2
	\end{flalign*}
	which implies $ \| d_j \| \ge \TheoryAlgConst^{-1/2} L^{-1/2} \varepsilon_k ^{1/2} = \SPTtwo ^ {-1} \omega_1 (\SPTtwo \omega_1 ^ {-1} \TheoryAlgConst^{-1/2} L^{-1/2} \varepsilon_k ^{1/2}) = \omega_1 \SPTtwo ^ {-1} \dminInd{k}$.
	\editTwo{Thus, we have established the premise of \Cref{claim:generic-direction-lower-bound} with $\GenericDmin = \dminInd{k}$, which 
	immediately implies that \Cref{lem:radius-and-direction-bounds}.\ref{lem:radius-and-direction-bounds:large-gradient}
	holds.}
\end{proof}

Now, we define specific indices that will be useful for our analysis. 
\edit{In particular,} 
\[
\pi(i) := \begin{cases}
	1 & i = 1 \\
	\min ~ \{ j \in \N : \varepsilon_{j} < \varepsilon_{\pi(i-1)} / 2  \} \cup \{ \infty \}  & i > 1
\end{cases}
\]  
\edit{which represents the first index such that $\varepsilon_{j}$ has decreased by a factor of two since index $\pi(i-1)$. Thus, $\varepsilon_{\pi(i)} \le 2^{1-i} \varepsilon_1$. We also define} 
\[
\tau(i) := \min ~  \{ j \in \N : j \ge \pi(i) ~\&~ \| d_j \| \ge \dminInd{\pi(i)}  \} \cup \{ \pi(i+1)\}
\]
\edit{which represents the first index following $\pi(i)$ with \editTwo{search direction norm} greater than or equal to $\dminInd{\pi(i)}$}.
We note that these definitions do not exclude the possibility that $\pi(i)$ or $\tau(i)$ are infinite. 
However, in our proofs, we will show that they are finite. 

Using that $\SPTtwo \omega_2 > 1$ from the requirements for Algorithm~\ref{alg:CAT} and that $j + 1 \in  [\pi(i), \pi(i+1) - 1] \cap \N \implies \varepsilon_{j+1} \ge \varepsilon_{\pi(i)} / 2$ by the definition of $\pi$, we can rewrite Lemma~\ref{lem:radius-and-direction-bounds}.\ref{lem:radius-and-direction-bounds:large-gradient} with $k = \pi(i)$ in terms of $\pi$ and $\tau$ as
\begin{subequations}
	\begin{flalign}\label{eq:min-growing-direction-and-lower-bound}
		\|d_{j}\| &\geq \min\{ \SPTtwo (\SPTtwo \omega_2) ^ {j - \pi(i)} r_{\pi(i)} , \dminInd{\pi(i)}  \} &  \forall j \in [\pi(i), \pi(i+1)-2] \\
		\|d_{j}\| &\geq  \dminInd{\pi(i)} & \forall j \in [\tau(i), \pi(i+1)-2] \label{eq:lower-bounded-direction-v2}
	\end{flalign}
\end{subequations}
which we will find useful for our proofs.

Using the definition of $\pi(i)$ and $\tau(i)$, we define the following \editTwo{repeating} phases. \editTwo{The first phase considers iterations between $\pi(i)$ and $\tau(i)$. During this phase, the lower bound in \eqref{eq:min-growing-direction-and-lower-bound} grows toward $\dminInd{\pi(i)}$; the phase ends when the search direction norm first reaches this threshold or when $\varepsilon_k$ decreases by a factor of two.} Consequently, we can bound the number of iterations in phase one (Lemma~\ref{lem:increase-radius-index-small-distance}).
The second phase considers iterations between $\tau(i)$ and $\pi(i+1)$ where the search direction remains above the threshold $\dminInd{\pi(i)}$, i.e., due to \eqref{eq:lower-bounded-direction-v2}. \editTwo{Thus, each successful step in this phase, i.e., each step with $\hat{\rho}_k \ge \beta$, reduces the function value by at least $\beta (-\Mk(d_k) + (\theta / 2) \min\{\|\grad f(x_k)\|, \|\grad f(x_k + d_k) \| \} \| d_k \| ) \ge \beta \theta \varepsilon_{\pi(i)} \dminInd{\pi(i)} / 4$.}
Next, using the results from Lemma~\ref{lem:radius-and-direction-bounds}.\ref{lem:radius-and-direction-bounds:part-general-radius-bounds-part2}, we argue that, up to a logarithmic term, the number of unsuccessful steps inside this interval cannot exceed the number of successful steps.
This allows us to bound the number of iterations in phase two (Lemma~\ref{lem:main-fully-adaptive-result}).
The iterates of Algorithm~\ref{alg:CAT} will alternate between phase one and phase two repeatedly until termination. However, according to the definition of $\tau(i)$, the iterates may transition directly to phase two without passing through phase one, i.e., $\tau(i) = \pi(i)$.

To bound the total number of iterations to terminate, we show that there exists $N \in \N$ such that $\varepsilon_{\pi(N+1)} \le \epsilon \le \varepsilon_{\pi(N)}$. Thus, the total number of iterations
until termination with $\varepsilon_k \le \epsilon$, where $\epsilon > 0$ is the termination tolerance, is at most 
\begin{flalign}\label{eq:decompose-piN+1}
	\pi(N+1)  = \pi(1) +  \sum_{i=1}^{N} \pi(i+1) - \tau(i) + \tau(i) - \pi(i).
\end{flalign}
To this inequality, we employ our bounds  on $\tau(i) - \pi(i)$ and $\pi(i + 1) - \tau(i)$ given in Lemma~\ref{lem:increase-radius-index-small-distance} and Lemma~\ref{lem:main-fully-adaptive-result}  to provide a bound on the total number of iterations (Theorem~\ref{thm:main-fully-adaptive-result}). 
\begin{lemma}\label{lem:increase-radius-index-small-distance}
		Suppose \edit{\Cref{assume:nice-directions} holds}, then for all $i \in \N$ such that $\pi(i) < \infty$ we have $\tau(i) < \infty$, and
	\begin{flalign}\label{eq:bound-recovery-phase}
		\tau(i) - \pi(i) \le \log_{\SPTtwo \omega_2} \left(\omega_2^{3i} \max\left\{\frac{\dminInd{\pi(i)}}{r_1}, 1 \right\}\right) .
	\end{flalign}
\end{lemma}

\begin{proof}
	First consider the case that $\tau(i) \le \pi(i) + 1$, then immediately we get \eqref{eq:bound-recovery-phase} because $\omega_2, i \ge 1$ and $\SPTtwo \omega_2 > 1$. 
	On the other hand, if $\pi(i) + 2 \le \tau(i)$ then, \editTwo{for any $t \in [\pi(i) + 2, \tau(i)]$} 
	\[
	\dminInd{\pi(i)} \undergreater{(a)} \|d_{\editTwo{t} - 2}\| 
	\underge{(b)} \SPTtwo (\SPTtwo \omega_2) ^ {\editTwo{t} - 2 - \pi(i)} r_{\pi(i)}
	\underge{(c)} (\SPTtwo \omega_2) ^ {\editTwo{t} - 2 - \pi(i)} \omega_1 ^ {-1} r_{\pi(i)}
	\]
	where inequality $(a)$ uses the definition of $\tau(i)$, inequality $(b)$ uses \eqref{eq:min-growing-direction-and-lower-bound} with $j=\editTwo{t}-2$, and inequality $(c)$ uses $\SPTtwo > \omega_1 ^ {-1}$ from the requirements of Algorithm~\ref{alg:CAT}.
	
	Rearranging the latter inequality and then taking the base $\SPTtwo \omega_2$ log of both sides gives 
	\begin{flalign}\label{eq:bound-recovery-phase-claim}
	\begin{split}
		\editTwo{t} - \pi(i) &\leq 2 + \log_{\SPTtwo \omega_2}\left( \omega_1 \frac{\dminInd{\pi(i)}}{r_{\pi(i)}} \right) \editTwo{= \log_{\SPTtwo \omega_2}\left( (\SPTtwo \omega_2)^2 \omega_1 \frac{\dminInd{\pi(i)}}{r_{\pi(i)}} \right)} \\
		&\le \log_{\SPTtwo \omega_2}\left( \omega_2^{3} \frac{\dminInd{\pi(i)}}{r_{\pi(i)}} \right) 
	\end{split}
	\end{flalign}	
	where the last transition uses that $\omega_2 \ge \omega_1$ and $\SPTtwo \in (0,1]$.
	Next, for all $\ell \le i$ consider the induction hypothesis that
	\begin{flalign}\label{eq:lowe-bound-radius-pi}
		\omega_1^{3(1-\ell)} \min\{r_1, \dminInd{\pi(i)}\} \leq r_{\pi(\ell)}.
	\end{flalign}
	Note that substituting \eqref{eq:lowe-bound-radius-pi} with $\ell = i$ into \eqref{eq:bound-recovery-phase-claim} and using $\omega_2 \ge \omega_1$ gives us \eqref{eq:bound-recovery-phase} as desired, thus it remains to prove that \eqref{eq:lowe-bound-radius-pi} holds.
	The hypothesis holds for $\ell=1$ because using the definition of $\pi$, we have $\pi(1) = 1$ and so $r_{\pi(\ell)}  = r_{\pi(1)} = r_1 \geq \min\{r_1, \dminInd{\pi(i)}\}$. Next, we suppose the hypothesis holds for $\ell = n$. 
	If $\pi(\ell) + 1 \ge \pi(\ell+1)$ then $r_{\pi(\ell+1)} \ge r_{\pi(\ell)} / \omega_1$ so 
	the hypothesis holds for $\ell = n+1$.
	If $\pi(\ell)  + 2 \le \pi(\ell+1)$ then we have  
	\begin{flalign*} 
		r_{\pi(\ell + 1)} &\underge{(a)}\frac{r_{\pi(\ell + 1)-2}}{\omega_1 ^ 2} \ge \frac{\|d_{\pi(\ell+1) - 2}\|}{\omega_1 ^ 2}  \underge{(b)} \omega_1^{-2} \min\{ \SPTtwo (\SPTtwo \omega_2) ^ {\pi(\ell+1) - 2 - \pi(\ell)} r_{\pi(\ell)}, \dminInd{\pi(\ell)}\} \\
		& \underge{(c)} \omega_1^{-2} \min\{ \SPTtwo   r_{\pi(\ell)} , \dminInd{\pi(\ell)} \}  \underge{(d)} \omega_1^{-3} \min\{r_{\pi(\ell)} , \dminInd{\pi(\ell)} \} \underge{(e)} \omega_1^{-3} \min\{r_{\pi(\ell)} , \dminInd{\pi(i)} \} \\ &\underge{(f)} \omega_1^{-3\ell} \min\{r_1, \dminInd{\pi(i)}\} =  \omega_1^{3(1-(\ell+1))} \min\{r_1, \dminInd{\pi(i)}\}
	\end{flalign*}
	where  inequality $(a)$ uses the update rule for $r_{k}$,  inequality $(b)$ uses \eqref{eq:min-growing-direction-and-lower-bound} with $j = \pi(\ell+1) - 2$, $(c)$ uses that $\pi(\ell)  + 2 \le \pi(\ell+1)$ and $\SPTtwo \omega_2 \ge \SPTtwo \omega_1 > 1$ by definition of Algorithm~\ref{alg:CAT}, $(d)$ uses that $\SPTtwo \geq \omega_1 ^ {-1}$ and $\omega_1 > 1$, $(e)$ uses $\dminInd{\pi(i)} \leq \dminInd{\pi(\ell)}$ because $\varepsilon_{\pi(i)} \leq \varepsilon_{\pi(\ell)}$ and $(f)$ employs the induction hypothesis. By induction \eqref{eq:lowe-bound-radius-pi} holds as desired.
\end{proof}

\begin{lemma}\label{lem:main-fully-adaptive-result}
	Suppose \edit{\Cref{assume:nice-directions} holds}, then for all $i \in \N$ such that $\tau(i) < \infty$ we have $\pi(i+1) < \infty$, and for all $t \in [\tau(i), \pi(i+1) - 2]$ we have
	\[
	t + 2 - \tau(i) \leq \log_{\omega_1}\left( \frac{\omega_2^3 \max\{ r_1, 2 \dmaxInd{\pi(i)} \}}{ \dminInd{\pi(i)}} \right) +  4 \log_{\omega_1}(\omega_1 \omega_2) \cdot \frac{ f(x_{\pi(i)}) - f(x_{t})}{ \beta \theta \varepsilon_{\pi(i)} \dminInd{\pi(i)} }.
	\]
\end{lemma}

\begin{proof}
	Note that if $t \le \tau(i) + 1$ the result trivially holds because $\omega_2 > \omega_1 > 1$. Thus throughout the proof we will assume $t \ge \tau(i) + 2$.
	Lemma~\ref{lem:radius-and-direction-bounds}.\ref{lem:radius-and-direction-bounds:part-general-radius-bounds-part2} with $k = \tau(i)$ implies that $r_t \leq r_{\tau(i)} \omega_1 ^ {-n_t} \omega_2 ^ {p_t}$ where $p_t = \abs{ \{ m \in [\tau(i), t) : \hat{\rho}_m \geq \beta \}}$ and $n_t = \abs{\{m \in [\tau(i), t) : \hat{\rho}_m < \beta \}}$.
	Taking the base $\omega_1$ log of both sides of this inequality and rearranging gives
	\[
	n_t = n_t \log_{\omega_1}(\omega_1) \le \log_{\omega_1}(r_{\tau(i)} / r_t) + p_t \log_{\omega_1}(\omega_2),
	\]
	applying \eqref{eq:lower-bounded-direction-v2} with $j = t$ to this inequality yields
	\begin{flalign}\label{eq:upper-bound-n-j} 
		n_t \le \log_{\omega_1}(r_{\tau(i)} / \dminInd{\pi(i)}) + p_t \log_{\omega_1}(\omega_2). 
	\end{flalign}
	Furthermore, if $\tau(i) = 1$ then $r_{\tau(i)} \le r_1$;
	on the other hand, if $\tau(i) > 1$ we have 
	\begin{flalign}\label{eq:upper-bound-r-tau-i}
		r_{\tau(i)} \underle{(a)} \omega_2r_{\tau(i) - 1} \underle{(b)} \omega_2\max\{r_1, \dmaxInd{\tau(i)} \} \underle{(c)} \omega_2\max\{ r_1, 2 \dmaxInd{\pi(i)} \}
	\end{flalign}
	where inequality $(a)$ uses the radius update rule of Algorithm~\ref{alg:CAT}, inequality (b) uses Lemma~\ref{lem:radius-and-direction-bounds}.\ref{lem:radius-and-direction-bounds:part-general-radius-bounds-part1} with  $k=1$ and $j=\tau(i)-1$, and inequality (c) uses the definition of $\dmaxInd{}$ and that $\varepsilon_{\tau(i)} \ge \varepsilon_{\pi(i)} / 2$ because $\tau(i) \le t \le \pi(i+1) - 2$.
	Substituting \eqref{eq:upper-bound-n-j} and \eqref{eq:upper-bound-r-tau-i} into $t + 2 - \tau(i) \le 2 + n_t + p_t$ (which holds by definition of $n_t$ and $p_t$) yields
	\begin{flalign}
		t + 2 - \tau(i) &\le 2 + \log_{\omega_1}\left(  \frac{\omega_2\max\{ r_1, 2 \dmaxInd{\pi(i)} \} }{ \dminInd{\pi(i)} } \right) + p_t \left( 1 + \log_{\omega_1}(\omega_2) \right) \notag \\
		&= \log_{\omega_1}\left(  \frac{\omega_1^2 \omega_2\max\{ r_1, 2 \dmaxInd{\pi(i)} \} }{ \dminInd{\pi(i)} } \right) + p_t  \log_{\omega_1}(\omega_1 \omega_2) .  \label{eq:j-minus-tau}
	\end{flalign}
	We proceed to bound $p_t$ since by \eqref{eq:j-minus-tau} it will allow us to bound $t + 2 - \tau(i)$.
	
	Suppose that the indices of $\{ m \in [\tau(i), t): \hat{\rho}_m \geq \beta \}$ are ordered \editTwo{in} increasing value by a permutation function $\kappa$, i.e., $\kappa(1) <\kappa(2) < \cdots < \kappa(p_t)$ with $\{ m \in [\tau(i), t): \hat{\rho}_m \geq \beta \} = \{ \kappa(\editTwo{m}) : \editTwo{m} \in [1, p_t] \cap \N \}$.
	Therefore, since the function values of the  iterates are nonincreasing we get
	\begin{flalign*}
		&f(x_{\tau(i)}) - f(x_{t}) \ge f(x_{\kappa(1)}) - f(x_{\kappa(p_{t})+1}) \ge \sum_{m=1}^{p_{t}} f(x_{\kappa(m)}) - f(x_{\kappa(m)+1}) \undereq{(a)} \notag \\
		& \sum_{m=1}^{p_t} \hat{\rho}_{\kappa(m)} \left( -\editTwo{\Mk[\kappa(m)]}({d_{\kappa(m)}}) + \frac{\theta}{2} \min \{ \| \grad f(x_{\kappa(m)})\|, \| \grad f(x_{\kappa(m)} +d_{\kappa(m)}) \|\} \| d_{\kappa(m)} \| \right) \\
		&\underge{(b)} \sum_{m=1}^{p_t} \beta \left( -\editTwo{\Mk[\kappa(m)]}(d_{\kappa(m)}) + \frac{\theta}{2} \min \{ \| \grad f(x_{\kappa(m)})\|, \| \grad f(x_{\kappa(m)} +d_{\kappa(m)}) \|\} \| d_{\kappa(m)} \| \right) \\
		&\underge{(c)} \frac{\beta \theta}{2}  \sum_{m=1}^{p_t} \min \{ \| \grad f(x_{\kappa(m)})\|, \| \grad f(x_{\kappa(m)} +d_{\kappa(m)}) \|\} \dminInd{\pi(i)} \\
		&\underge{(d)} \frac{\varepsilon_{\pi(i)} \beta \theta}{4}  p_t \dminInd{\pi(i)}
	\end{flalign*}
	where equality $(a)$ uses the definition of $\hat{\rho}_{\kappa(m)}$, inequality $(b)$ follows from $\hat{\rho}_{\kappa(m)} \ge \beta$, inequality $(c)$ uses $-\editTwo{\Mk[\kappa(m)]}(d_{\kappa(m)}) \ge 0$ and \eqref{eq:lower-bounded-direction-v2}, and inequality $(d)$ uses \edit{ the fact that } $\min \{ \| \grad f(x_{\kappa(m)})\|, \| \grad f(x_{\kappa(m)} +d_{\kappa(m)}) \|\} \ge \frac{\varepsilon_{\pi(i)}}{2}$ by $\kappa(m) < t \le \pi(i+1) - 2$.
	Rearranging the latter inequality for $p_t$ using the fact that $\beta \theta \varepsilon_{\pi(i)} \dminInd{\pi(i)} > 0$ yields 
	\[
	p_{t} \le \frac{4 (f(x_{\tau(i)}) - f(x_{t}))}{ \beta \theta \varepsilon_{\pi(i)} \dminInd{\pi(i)} }.
	\]
	Using this inequality and \eqref{eq:j-minus-tau} we get
	\begin{flalign}\label{eq:j-minus-tau-i-bound}
		t + 2 - \tau(i) &\le \log_{\omega_1}\left(  \frac{\omega_1^2 \omega_2\max\{ r_1, 2 \dmaxInd{\pi(i)} \} }{ \dminInd{\pi(i)} } \right) + \frac{4 (f(x_{\tau(i)}) - f(x_{t}))}{ \beta \theta \varepsilon_{\pi(i)} \dminInd{\pi(i)} }  \log_{\omega_1}(\omega_1 \omega_2).
	\end{flalign}
	Since $f(x_{t})$ is bounded below we deduce that $t$ is bounded above and thus $\pi(\editTwo{i+1}) < \infty$.
	We get the result by using $\omega_1 \leq \omega_2$ and observing that $f(x_{\tau(i)}) \le f(x_{\pi(i)})$ as $\pi(i) \le \tau(i)$.
\end{proof}

\newcommand{\FinalConst}{\hat{C}} 

We now provide our convergence \edit{guarantee} for Algorithm~\ref{alg:CAT}.

\begin{theorem}\label{thm:main-fully-adaptive-result}
	Suppose \edit{\Cref{assume:nice-directions} holds}, then for all $\epsilon \in (0,\varepsilon_1)$ there exists some iteration $k$ with $\varepsilon_k \le \epsilon$ and
	\begin{flalign*}
		k 
		\le \FinalConst \cdot \frac{\Delta_f L ^ {1/2}}{\epsilon ^ {3/2}} + 
		\log_{2}\left(\frac{2 \varepsilon_1}{\epsilon} \right) \log_{\omega_1}\left(  \editTwo{\omega_1 +} \frac{\omega_1 \omega_2^3 \TheoryAlgConst^{1/2} L^{1/2 }}{\SPTtwo} \cdot \max\left\{ \frac{r_1}{\epsilon ^ {1/2}}, \frac{4 \omega_2}{\beta \theta} \cdot \frac{\Delta_f}{\epsilon ^ {3/2}} \right\} \right) \notag \\ 
		+ \log_{2}\left(\frac{2 \varepsilon_1}{\epsilon} \right) \left( 3 \log_{2}\left(\frac{2 \varepsilon_1}{\epsilon} \right)  \log_{\SPTtwo \omega_2} (\omega_2) +  \log_{\SPTtwo \omega_2} \left(\max\left\{\frac{\editTwo{\SPTtwo \omega_1^{-1} \varepsilon_1^{1/2}}}{\editTwo{\TheoryAlgConst^{1/2} L^{1/2}} r_1}, 1\right\} \right) \right) + 1
	\end{flalign*}
	where $\FinalConst := 4 \log_{\omega_1}(\omega_1\omega_2)\frac{\TheoryAlgConst ^ {1/2} \omega_1}{\beta \theta \SPTtwo}$ and $\TheoryAlgConst$ are problem-independent constants, recalling from \eqref{eq:constant-algorithm-definition} that 
	\begin{flalign*}
		\TheoryAlgConst = 
		\max\left\{\frac{1}{3\AlgConst}, \frac{2 + 3 \SPTthree (1-\beta)}{3\left(\SPTthree (1 - 2 \SPTone) (1-\beta) - \beta \theta \right)} \right\}.
	\end{flalign*}	
\end{theorem}

\begin{proof}
	First we show that $\pi(i)$ and $\tau(i)$ are finite for all $i \in \N$. Using the results of Lemma~\ref{lem:increase-radius-index-small-distance}, we know if $\pi(i)$ is finite then $\tau(i)$ is finite and hence using the results of Lemma~\ref{lem:main-fully-adaptive-result}, we get $\pi(i+1)$ is finite since $\tau(i)$ is finite. 
	Also, $\pi(1)$ is finite by definition.
	Therefore by induction we deduce that  $\pi(i)$ and $\tau(i)$ are finite.
	
	By Lemma~\ref{lem:increase-radius-index-small-distance} we have
	\begin{flalign}
		\sum_{i = 1} ^ {N} \tau(i) - \pi(i)  &\le \sum_{i = 1} ^ {N} \log_{\SPTtwo \omega_2} \left(\omega_2^{3i} \max\left\{\frac{\dminInd{\pi(i)}}{r_1}, 1 \right\}\right)  \notag  \\
		&\le N \log_{\SPTtwo \omega_2} \left(\omega_2^{3 N} \max\left\{\frac{\editTwo{\dminInd{1}}}{r_1}, 1 \right\}\right) 	\label{eq:tau-i-pi-i-bound}
	\end{flalign}
	where the second inequality uses $\dminInd{\pi(i)} \leq \editTwo{\dminInd{1}}$ based on the definition of $\dminInd{}$.
	
	Lemma~\ref{lem:main-fully-adaptive-result} with $t = \pi(i + 1) - 2$ implies that \editTwo{for $S := \{ i \in \N \cap [1,N] : \pi(i+1) - \tau(i) >  1 \}$},
	\begin{flalign}
		&\sum_{\editTwo{i \in S}} \pi(i + 1) - \tau(i) \notag \\
		& \underle{(a)} \sum_{i=1}^{N} \log_{\omega_1}\left( \frac{\omega_2^3 \max\{ r_1, 2 \dmaxInd{\pi(i)} \}}{ \dminInd{\pi(i)}} \right) +  4 \log_{\omega_1}(\omega_1 \omega_2) \cdot \frac{ f(x_{\pi(i)}) - f(x_{\pi(i + 1)})}{ \beta \theta \varepsilon_{\pi(i)} \dminInd{\pi(i)} } \notag \\
		&\underle{(b)} N \log_{\omega_1}\left( \frac{\omega_2^3 \max\{ r_1, 2\dmaxInd{\pi(N)} \}}{\dminInd{\pi(N)}} \right) + \frac{4 \log_{\omega_1}(\omega_1 \omega_2)}{\beta \theta \varepsilon_{\pi(N)} \dminInd{\pi(N)}} \sum_{i=1}^{N} f(x_{\pi(i)}) - f(x_{\pi(i+1)}) \notag \\
		&\underle{(c)} N \log_{\omega_1}\left( \frac{\omega_2^3 \max\{ r_1, 2\dmaxInd{\pi(N)} \}}{\dminInd{\pi(N)}} \right) +  \frac{4 \log_{\omega_1}(\omega_1 \omega_2)}{\beta \theta \varepsilon_{\pi(N)} \dminInd{\pi(N)}} (f(x_{1}) - f(x_{\pi(N+1)})) 
		\label{eq:sum:lem:main-fully-adaptive-result}
	\end{flalign}
	where inequality $(a)$ uses $f(x_{\pi(i+1)-2}) \ge f(x_{\pi(i + 1)})$, inequality $(b)$ uses that $\varepsilon_{\pi(N)} \le \varepsilon_{\pi(i)}$ 
	and thus
	$\dminInd{\pi(N)} \le \dminInd{\pi(i)}$ and $\dmaxInd{\pi(i)} \le \dmaxInd{\pi(N)}$, and $(c)$ is by telescoping. 
	
	Combining \eqref{eq:tau-i-pi-i-bound} and \eqref{eq:sum:lem:main-fully-adaptive-result} and using \eqref{eq:decompose-piN+1} \editTwo{and $\sum_{i \in [1,N] \cap \N : i \not\in S} \pi(i+1) - \tau(i) \le N$} we get 
	\begin{flalign}\label{eq:J-eps}
		\pi(N+1) &\le  \frac{4 \log_{\omega_1}(\omega_1 \omega_2)}{\beta \theta \varepsilon_{\pi(N)} \dminInd{\pi(N)}} (f(x_{1}) - f(x_{\pi(N+1)})) + N \log_{\SPTtwo \omega_2} \left(\omega_2^{3 N} \max\left\{\frac{\editTwo{\dminInd{1}}}{r_1}, 1 \right\}\right) +  \notag \\
		&N \log_{\omega_1}\left( \editTwo{\omega_1} + \frac{\omega_2^3 \max\{ r_1, 2\dmaxInd{\pi(N)} \}}{\dminInd{\pi(N)}} \right)
		 + 1.
	\end{flalign}
	Next, by definition of $\pi(N)$ we have $\epsilon \le \varepsilon_{\pi(N)} \le \varepsilon_{\pi(N-1)} / 2 \le \varepsilon_{\pi(1)} 2^{1-N} = \varepsilon_{1} 2^{1-N}$. Rearranging for $N$ gives 
	\begin{flalign}\label{eq:upper-bound-N}
		N \le 1 + \log_2(\varepsilon_1 / \epsilon ) = \log_2(2 \varepsilon_1 / \epsilon ).
	\end{flalign}
	Since $\epsilon \le \varepsilon_{\pi(N)}$ we can substitute into \eqref{eq:J-eps} that 
	\editTwo{$\dminInd{1} = \SPTtwo \omega_1 ^ {-1}\TheoryAlgConst^{-1/2} L^{-1/2} \varepsilon_{1}^{1/2}$},
	$\dmaxInd{\pi(N)} = \frac{2 \omega_2}{\beta \theta} \cdot \frac{\Delta_f}{\varepsilon_{\pi(N)}} \le \frac{2 \omega_2}{\beta \theta} \cdot \frac{\Delta_f}{\epsilon}$ and 
	$\dminInd{\pi(N)} \ge \SPTtwo \omega_1 ^ {-1}\TheoryAlgConst^{-1/2} L^{-1/2} \epsilon^{1/2} $
	which combined with \eqref{eq:upper-bound-N} gives 
	Theorem~\ref{thm:main-fully-adaptive-result}. 

\end{proof}

\newcommand{\kWarm}{k_{\rm warm}^{\epsilon}}
\newcommand{\dLowerFO}{\ubar{d}_{\epsilon}^G}

\editTwo{
\section{Convergence on functions with Lipschitz first derivatives}\label{sec:lipschitz-first-derivatives}
This section provides guarantees on 
the convergence of our method on 
functions with $G$-Lipschitz gradients.
In particular, we prove that our method 
finds an $\epsilon$-approximate stationary point 
in at most $O(\Delta_f G \epsilon^{-2}) + \tilde{O}(1)$ iterations.
\Cref{assume-iterates-first-derivative-lipschitz} is the assumption under
 which we prove the result, which is slightly weaker than the function 
 being bounded below and having $G$-Lipschitz gradients.
 The main result from this section (\Cref{thm:lipschitz-first-derivatives}) will 
 also be useful for establishing convergence of the method on continuously twice-differentiable
 functions (\Cref{sec:locally-lipschitz-second-derivatives}).

\begin{assumption}\label{assume-iterates-first-derivative-lipschitz}
Suppose that there exists $G > 0$ such that for all $k \in \N$,
\begin{flalign}\label{eq:lipschitz-first-derivatives}
&\|\grad^2 f(x_k)\| \le G
\qquad\text{and}\qquad
f(x_k + d_k) \le f(x_k) + \grad f(x_k)^\top d_k + \frac{G}{2} \| d_k \|^2 . &&
\end{flalign}
\end{assumption}
Let
\begin{flalign}\label{eq:d-ubar-G-def}
\dLowerFO := \frac{\SPTtwo \chi}{\omega_1} \cdot \frac{\epsilon}{G} \text{  where  } \chi := \frac{\SPTthree(1-\beta)(1-\SPTone)-\beta\theta}
		{2+\SPTthree(1-\beta)}.
\end{flalign}
\Cref{lemma:first-derivatives} provides a lower bound on the search direction norm that grows until it reaches $\dLowerFO$ which thereafter represents the smallest search direction size possible.
The proof involves establishing the premise of \Cref{claim:generic-direction-lower-bound}
using \Cref{assume-iterates-first-derivative-lipschitz}.

\begin{lemma}
	\label{lemma:first-derivatives}
	Suppose \Cref{assume-iterates-first-derivative-lipschitz} holds 
	and $\| \grad f(x_j) \| \ge \epsilon$ for all $j \in [1,K]$.
	Then, for all $j \in [1,K]$
	\[
	\| d_j \| \ge \min\left\{ r_1 \SPTtwo(\omega_2 \SPTtwo)^{j-1}, \dLowerFO \right\}.
	\]
\end{lemma}

\begin{proof}
The proof proceeds by assuming $\|d_k\|\le \chi\frac{\| \grad f(x_k) \|}{G}$
and showing this implies that $\delta_k \| d_k \| > 0$ and the step is successful, i.e., $\hat{\rho}_k\ge\beta$.
This implies that \Cref{eq:generic-direction-lower-bound-premise}
holds with $\GenericDmin = \SPTtwo \chi\frac{\epsilon}{G \omega_1}$ and therefore by \Cref{claim:generic-direction-lower-bound}
our desired result holds.

In particular, using $(a)$ $\grad\Mk(d_k)=\grad f(x_k)+\grad^2 f(x_k)d_k$, the triangle inequality, and
the subproblem termination tolerance on the gradient of the model given in \eqref{eq:model-gradient-weaker},
$(b)$ \Cref{assume-iterates-first-derivative-lipschitz},
$(c)$
$\|d_k\|\le\chi \|\grad f(x_k)\|/G$ 
and $(d)$ the definition of $\chi$ and the assumptions on the parameters in \Cref{alg:CAT}, we obtain
\begin{flalign}\label{eq:Lip-first:delta-times-dir-bound}
\begin{aligned}
\delta_k\|d_k\| &= \| \delta_k d_k + \grad \Mk(d_k) - \grad \Mk(d_k) \| \\
&\underge{(a)}  \| \grad f(x_k) \| - \| \grad^2 f(x_k)d_k \| - \SPTone \varepsilon_k \underge{(b)} \|\grad f(x_k)\|-G\|d_k\|-\SPTone\varepsilon_k \\
&\underge{(c)}  \|\grad f(x_k)\| (1-\SPTone-\chi) \undergreater{(d)} 0.
\end{aligned}
\end{flalign}	
Next, by $(a)$ $\|d_k\|\le \chi \| \grad f(x_k) \| /G$, $(b)$ \Cref{eq:Lip-first:delta-times-dir-bound}, $(c)$ the definition of
$\chi$, and $(d)$ \Cref{eq:model-reduction-by-dist1} we get
\begin{flalign}\label{eq:error-absorbed-first-derivatives}
\begin{aligned}
	G\|d_k\|^2
	+
	\frac{\beta\theta}{2}
	\|\grad f(x_k)\|\|d_k\| &\underle{(a)}  \left( \chi + \frac{\beta\theta}{2} \right)
	\|\grad f(x_k)\|\|d_k\| \\
	&\underle{(b)} \left( \chi + \frac{\beta\theta}{2} \right)
	\frac{\delta_k \|d_k\|^2}{(1-\SPTone-\chi)} \\
	&\undereq{(c)}
	\SPTthree(1-\beta)\frac{\delta_k}{2}\|d_k\|^2 \underle{(d)}
	(1-\beta)(-\Mk(d_k)).
\end{aligned} 
\end{flalign}
Finally, by $(a)$ \Cref{eq:lipschitz-first-derivatives} and the definition of $\Mk(d_k)$, and $(b)$
\Cref{eq:error-absorbed-first-derivatives},
\begin{flalign*}
f(x_k)-f(x_k+d_k)
&\underge{(a)}
-\Mk(d_k)-G\|d_k\|^2 \\
&\underge{(b)}
\beta
\biggl(
	-\Mk(d_k)
	+
	\frac{\theta}{2}
	\min\{\|\grad f(x_k)\|,\|\grad f(x_k+d_k)\|\}\|d_k\|
\biggr).
\end{flalign*}
Thus, $\hat{\rho}_k\ge\beta$.
\end{proof}

Given \Cref{lemma:first-derivatives}, we can use arguments very similar to those in \citet[Theorem 1]{hamad2022consistently} to prove \Cref{thm:lipschitz-first-derivatives}. In particular, successful steps with $\|d_k\|\ge\dLowerFO$ yield significant function value reduction through the $\min\{\|\grad f(x_k)\|, \|\grad f(x_k + d_k) \| \} \| d_k \|$ term in \Cref{eq:rho-hat-k}.

\begin{theorem}\label{thm:lipschitz-first-derivatives}
	Suppose \Cref{assume-iterates-first-derivative-lipschitz} holds and
	$f(x_1)-f(x_k)\le \Delta_f$ for all $k\in\N$. Recall from
	\Cref{eq:d-ubar-G-def} that $\chi > 0$ is a problem-independent constant and $\dLowerFO = \Theta(\epsilon/G)$. Then, for all
	$\epsilon\in(0,\varepsilon_1)$, there exists an iteration $k$ with
	$\varepsilon_k \le\epsilon$ where $k$ is less than or equal to
	\begin{flalign*}
		\editTwo{2} + \frac{\editTwo{2\omega_1} \log(\omega_2 \omega_1)}{\editTwo{\beta\theta\SPTtwo\chi\log(\omega_1)}} \cdot \frac{\Delta_f G}{\epsilon^2} + \frac{\log(r_1 / \dLowerFO)}{\log(\omega_1)} + \frac{\log(\omega_2 \omega_1)}{\log(\omega_1)} \cdot 
\left\lceil 1 + \frac{\log\left(1+\dLowerFO/(r_1\SPTtwo)\right)}{
\log\left(\omega_2\SPTtwo\right)}\right\rceil.
	\end{flalign*}
\end{theorem}

\begin{proof}
\editTwo{Suppose that 
$\varepsilon_{j+1} \ge \epsilon$ for some $j \in \N$.}
Let $p_j := \abs{ \{ m \in [1, j) : \hat{\rho}_m \geq \beta \}}$
and $n_j := \abs{\{m \in [1, j) : \hat{\rho}_m < \beta \}}$.
By \Cref{lem:radius-and-direction-bounds}.\ref{lem:radius-and-direction-bounds:part-general-radius-bounds-part2}
we have $r_j \leq r_{1} \omega_1 ^ {-n_j} \omega_2 ^ {p_j}$. Rearranging this inequality gives
\[
n_j \leq \frac{\log(r_1 / r_j) + p_j \log(\omega_2)}{\log(\omega_1)}.
\]
Thus, to bound $p_j + n_j$ it suffices to bound $p_j$.
Next, from \Cref{lemma:first-derivatives} we have 
$\| d_j \| \ge \min\left\{ r_1 \SPTtwo(\omega_2 \SPTtwo)^{j-1}, \dLowerFO \right\}$.
Define
\[
\kWarm := 1+ \left\lceil
\frac{
\log\left(1+\dLowerFO/(r_1\SPTtwo)\right)
}{
\log\left(\omega_2\SPTtwo\right)
}
\right\rceil.
\]
If $j\ge\kWarm$, then $\|d_j\|\ge\dLowerFO$.
Let $\mathcal{P} := \{ m \in [\kWarm, j) : \hat{\rho}_m \geq \beta \}$ be the set of successful
steps after the warm-up period:
\begin{flalign*}
f(x_1) - f(x_j) &\ge f(x_{\kWarm}) - f(x_j)
\ge \sum_{i \in \mathcal{P}}\left(f(x_i) - f(x_{i+1})\right) \\
&\ge \frac{\beta\theta}{2}\sum_{i\in\mathcal{P}}
\min\{\|\grad f(x_i)\|,\|\grad f(x_{i+1})\|\}\|d_i\| \\
&\ge \frac{\beta\theta}{2}\dLowerFO\epsilon|\mathcal{P}|
\ge \frac{\beta\theta}{2}\dLowerFO\epsilon(p_j-\kWarm).
\end{flalign*}
It follows that $p_j-\kWarm\le 2\Delta_f/(\beta\theta\dLowerFO\epsilon)$ and therefore, for $j\ge\kWarm$,
\[
j-1 = p_j + n_j \le \frac{\log(r_1 / \dLowerFO)}{\log(\omega_1)} + \frac{\log(\omega_2 \omega_1)}{\log(\omega_1)} \cdot \left( \frac{2\Delta_f}{\beta\theta\dLowerFO\epsilon} + \kWarm \right).
\]
Substituting
$\dLowerFO := \frac{\SPTtwo \chi}{\omega_1} \cdot \frac{\epsilon}{G}$ gives the desired result.
\end{proof}
}
\section{Convergence on \editTwo{continuously twice-differentiable} functions} \label{sec:locally-lipschitz-second-derivatives}

This section establishes the convergence properties of our algorithm \editTwo{under only the assumption that the} 
function \editTwo{is continuously twice-differentiable}. \editTwo{In \Cref{thm:local-Lipschitz},} we show that \editTwo{for the iterates produced by our algorithm either $\lim_{k \rightarrow \infty} f(x_k) = -\infty$ or 
$\lim_{k \rightarrow \infty} \varepsilon_k = 0$}, even when the sequence of 
iterates is unbounded.

Existing convergence results for optimization methods in the literature are vacuous if the iterates diverge (\cite[Proposition~1.2.1]{bertsekas_nonlinear_book}) 
or assume Lipschitz \editTwo{first or second} derivatives \editTwo{ either } on the initial level set (see \cite{nesterov2006cubic, nocedal2006numerical, armijo1966minimization, powell2010convergence})
, $\{ x \in \R^{n} : f(x) \le f(x_1) \}$ \editTwo{or on an open convex set containing the iterates \cite[Theorem 3.14]{curtis2017trust}}. 

These conditions can be \editTwo{impossible} to establish\editTwo{, even when the function is infinitely differentiable and bounded below}. 
\editTwo{In particular, consider the function
\[
f(x) := \exp(-[x]_1) + [x]_1^2 (\exp([x]_2) - 1 - [x]_2) + [x]_2^2.
\]
For this function any sequence $\{ x_k \}$ with $\| \grad f(x_k) \| \rightarrow 0$ has $[x_k]_1 \rightarrow \infty$ and $[x_k]_2 \rightarrow 0$.
Therefore $\frac{\partial^p f(x_k)}{\partial [x]_2^p} \ge [x_k]_1^2 \exp([x_k]_2) \rightarrow \infty$ for all
$p \ge 2$. Consequently, any open
set containing a sequence of iterates with $\| \grad f(x_k) \| \rightarrow 0$ cannot have Lipschitz $p$th-order derivatives
for any $p\ge2$. Thus, for this example, it is impossible to establish the conditions required by existing convergence results 
\cite{bertsekas_nonlinear_book, nesterov2006cubic, nocedal2006numerical, armijo1966minimization, powell2010convergence, curtis2017trust}.
Similar issues can arise when minimizing the cross-entropy loss composed with an infinitely differentiable 
nonlinear model where there is no finite minimizer.
These issues can, of course, be resolved by adding a penalty to stop the iterates from diverging, but appropriately
deciding on this penalty may require trial and error.}

\editTwo{Key to proving \Cref{thm:local-Lipschitz} is that, on accepted steps, the function value falls by at least
$\theta \sigma \varepsilon_{k+1} \| d_k \| / 2$
because acceptance requires $\hat{\rho}_k \ge \AcceptableConst$ in \Cref{alg:CAT}.}

\begin{theorem}\label{thm:local-Lipschitz}
\editTwo{Suppose that $f$ is continuously twice-differentiable and $\AcceptableConst > 0$. Then
$\lim_{k \rightarrow \infty} \varepsilon_k = 0$ or $\lim_{k \rightarrow \infty} f(x_k) = -\infty$.}
\end{theorem}

\begin{proof}
Assume that $f(x_1) - f(x_k) \le \Delta_f < \infty$ for all $k \in \N$ since otherwise we have $\lim_{k \rightarrow \infty} f(x_k) = -\infty$.
Define the set of accepted steps $\mathcal{A}_K$ to be the set of indices corresponding to accepted steps $k \le K$, where an accepted step $k$ satisfies $f(x_k + d_k) \le f(x_k)$ and $\hat{\rho}_k \geq \AcceptableConst$ (see definition of accepted steps in \Cref{alg:CAT}). We now consider two cases: when the iterates are unbounded and when they are bounded.

\paragraph{Case One: Unbounded Iterates}
Consider the case where  the sequence of iterates is unbounded, specifically $\lim_{K \rightarrow \infty} \max_{i \in [K]} \| x_i \| = \infty$.
Then,
\begin{flalign*} 
	\notag
	\Delta_f &\ge \sum_{k \in \mathcal{A}_K} f(x_{k}) - f(x_{k+1})
	\underge{(a)} \frac{\theta \AcceptableConst \sum_{k \in \mathcal{A}_K} \min\{\|\grad f(x_{k})\|, \|\grad f(x_{k+1} ) \| \} \| d_{k} \|  }{2} \\
	&\underge{(b)} \frac{\theta \AcceptableConst \varepsilon_{K\editTwo{+1}} \sum_{k \in \mathcal{A}_K} \| d_{k} \|}{2} \undereq{(c)} \frac{\theta \AcceptableConst \varepsilon_{K\editTwo{+1}} \sum_{k=1}^K \| x_{k+1} - x_k \|}{2} \ge \frac{\theta \AcceptableConst  \varepsilon_{K\editTwo{+1}} \| x_{K+1} - x_{1} \|}{2}
\end{flalign*}  
where $(a)$ uses the definition of the actual-to-predicted reduction ratio $\hat{\rho}_k$, the condition $\hat{\rho}_k \ge \AcceptableConst$, the property that $-\Mk(d_{k}) \geq 0$, and the definition of $\editTwo{\epsilon_{K+1}}$, $(b)$ uses the nonincreasing nature of the sequence $\editTwo{\epsilon_{K}}$, and $(c)$ uses that rejected steps satisfy $x_{k+1} = x_k$.
Rearranging so that only $\editTwo{\epsilon_K}$ is on the RHS and using that $\lim_{K \rightarrow \infty} \max_{i \in [K]} \| x_i \| = \infty$ gives 
$\lim_{K \rightarrow \infty} \editTwo{\epsilon_K} = 0$
as desired.

\paragraph{Case Two: Bounded Iterates}

First note that:
\begin{flalign*}
r_{K+1} \underle{(a)} \max_{k \in \mathcal{A}_K} \{ \omega_2 \| d_k \|, r_1 \} \underle{(b)} \max_{k \in \mathcal{A}_K} \{ \omega_2 (\| x_k \| + \| x_{k+1} \|) , r_1 \}  
\end{flalign*}
where inequality $(a)$ follows by induction on $r_{k+1} \leq \max\{ \omega_2 \| d_k \|,  r_k \}$ for all $k \in  \mathcal{A}_K$ and  $r_{k+1} \le r_k$ for $k \not\in \mathcal{A}_K$  (these inequalities follow from the radius update rule in \Cref{alg:CAT}) and $(b)$ then follows from the triangle inequality.
We conclude that since the iterates are bounded the trust-region radius, $r_k$, is also bounded.

Thus, there exists some $R' > 0$ such that $x_k + d_k \in \mathcal{Q} := \{ x \in \mathbb{R}^{\editTwo{n}} : \| x \| \le R' \}$ for all $k \in \mathbb{N}$.
Since $\mathcal{Q}$ is a compact set and \editTwo{$f$ is continuously twice-differentiable}\editTwo{,} there exists some $\editTwo{G} > 0$
such that \editTwo{$\grad f$} is \editTwo{$G$}-Lipschitz continuous on $\mathcal{Q}$ which guarantees that \editTwo{\Cref{assume-iterates-first-derivative-lipschitz}} holds.
Since \editTwo{\Cref{assume-iterates-first-derivative-lipschitz}} holds we can apply \editTwo{\Cref{thm:lipschitz-first-derivatives}}
to deduce that \editTwo{$\lim_{k \rightarrow \infty} \varepsilon_k = 0$} as desired. 
\end{proof}

%
%

\section{Numerical results}\label{sec:numerical-results}

This section evaluates the effectiveness of our method compared with 
state-of-the-art solvers. We compare our method against the \editTwo{Fortran implementations} 
of a trust-region method (TRU) and adaptive cubic regularized Newton (ARC) that are 
available through the GALAHAD library \citep{gould2010solving}, 
\editTwo{and the universal trust-region (UTR) method \cite{jiang2023universal}, 
which we implement ourselves\footnote{\editTwo{We tried using the code provided by the authors \url{https://github.com/COPT-Public/DRSOM.jl/blob/main/src/algorithms/utr.jl} but this did not accurately implement the algorithm described in their paper. Implementing the code ourselves also has added benefit of giving a fairer comparision because all the primitives including the trust-region subproblem solver are the same.}}}.
Our method is implemented in an open-source Julia module available at 
\url{https://github.com/fadihamad94/CAT-Journal}. This repository also 
provides detailed tables of results and instructions for reproducing the experiments. 
The experiments \editTwo{were performed} in a  environment restricted to a single-thread 
on a Linux virtual machine that has an an \editTwo{AMD EPYC 7302 processor and $256$ GB of RAM}.

\editTwo{\paragraph{Subproblem solvers} 
For our experiments ARC and TRU used their factorization-based subproblem solvers, called RQS and TRS respectively, which are based on 
applying Newton's method to find a root of the secular equation. 
Initially we tried to use TRS but it had issues 
consistently satisfying our termination criteria, i.e., \eqref{eq:subproblem-termination-criteria}.
For this reason, we built a more reliable trust-region subproblem solver that is also a  pure factorization method based on a bisection routine (and inverse-power iteration for the hard-case). The full details of the implementation are described in Appendix~\ref{sec:Solving the trust-region sub-problem}.
In terms of linear solvers, our trust-region solver
uses the open source Cholesky factorization available in Julia; GALAHAD uses the Harwell Subroutine Library (HSL) package (MA57). 
}

\paragraph{Algorithmic parameters} The parameters for our method (unless explicitly mentioned) are: \edit{$\AcceptableConst = 0$}, $\beta = 0.1$, $\theta = 0.1$, $\omega_1 = 8.0$, $\omega_2 = 16.0$, $\SPTone = 0.01$, $\SPTtwo = 0.8$, and $\SPTthree = 0.5$. For the initial radius, we  include a more sophisticated radius
selection heuristic based on the initial starting point: $r_1 = \frac{10 \| \grad f(x_1)\|}{\| \grad ^ 2 f(x_1) \|}$. When implementing Algorithm \ref{alg:CAT} with some target tolerance $\epsilon$, we immediately terminate when $\varepsilon_k \leq \epsilon$.
\editTwo{
	For UTR we used the parameters suggested by the authors 
(using \emph{their} notation this means, $\rho_0 = 1$, $\rho_{\min} = 10^{-5}$, $\eta = 29/1200$,
$\xi=0.5$, $\gamma_1 = 2.25$ and $\gamma_2=1.75$).
We used the same trust-region subproblem termination criteria as our method (note their paper provides no theoretical 
guidance on the trust-region subproblem solver subproblem termination criteria).
For GALAHAD we used the default parameters.
}

\paragraph{Termination criteria} 
\editTwo{All methods terminate with a succesful status if they obtain a point with 
gradient norm less than $10 ^ {-5}$, i.e.,
$\varepsilon_{k} \le 10^{-5}$.
Methods are terminated with a failure status code if they exceed the iteration limit of $100,000$ or the time limit of five hours.
Other more niche failure statuses are described in the caption of \Cref{table:cutest-run-statistics-3}.
}

\paragraph{Benchmark problems} We evaluated the performance of our algorithm on the CUTEst benchmark set \citep{orban-siqueira-cutest-2020} available \edit{at} \url{https://github.com/JuliaSmoothOptimizers/CUTEst.jl}. This is a
standard test set for nonlinear optimization algorithms. We selected all the unconstrained problems with \editTwo{at least} $100$ variables, which gives us a total of $125$ problems.

\editTwo{
\paragraph{Benchmark results}
From Table~\ref{table:cutest-run-statistics-combined},
which records the shifted geometric mean,
shows our algorithm outperforms UTR, TRU and ARC across  of all metrics.
Compared with UTR,
it is at least $1.1\times$ faster in terms of the number of function evaluations,
$1.3\times$ faster in terms of number gradient evalutions, and $1.2\times$ faster both in terms of number Hessian evalutions and time. TRU and ARC are even slower in all these metrics, 
although this is likely driven primarily 
by reliability issues (\Cref{table:cutest-run-statistics-3}) which at significant cost to the geometric mean (a penalty of twice time limit for the time metric and twice the iteration limit for the others).

\Cref{fig:performance-fraction-problems-solved}
provides more granular information on the performance of the methods. In each plot,
we can see our method is typically the highest curve (which implies it is better).
The bottom-right of \Cref{fig:performance-fraction-problems-solved}
shows that our method obtains better terminal objective values than the other methods. 

}

A table with our full
results can be found at \url{https://github.com/fadihamad94/CAT-Journal/tree/master/results}.

\begin{table}[htbp]
	\centering
	\small
	\caption{\editTwo{Performance statistics 
	on the 125 unconstrained CUTEst problems with at least 100 variables. Failures are assigned twice the iteration limit and twice the time limit: $200,000$ \editTwo{evaluations} and $10$ hours\editTwo{,} respectively. Note that $\#f$, $\# \grad f$, and $\# \grad^2 f$ stand for the number of function, gradient\editTwo{,} and Hessian evaluations\editTwo{,} respectively. Additionally, $\#$ fact. is the number of Cholesky factorizations performed by the method. The conference version omits a factorization count because it uses SVD in the hard-case so factorizations are a poor performance measure.}}
	\label{table:cutest-run-statistics-combined}
	\setlength{\tabcolsep}{4.0pt}
	\editTwo{\input{table-content/main/cutest_performance_statistics.tex}}
\end{table}

\begin{table}[htbp]
	\centering
	\small
	\caption{\editTwo{Failure statistics 
	on the 125 unconstrained CUTEst problems with at least 100 variables. 
	Small step refers to when the search direction of the method gets so tiny (i.e., fall below $2 \times 10^{-16}$) that further progress is not plausible in which case the algorithm automatically terminates.}}
	\label{table:cutest-run-statistics-3}
	\label{label:failures}
	\setlength{\tabcolsep}{4.0pt}
	\editTwo{\input{table-content/main/cutest_failure_reasons.tex}}
\end{table}

\begin{figure}[htbp]
	\centering
	\caption{\editTwo{Performance comparison 
	on the 125 unconstrained CUTEst problems with at least 100 variables. The objective plot (bottom right) compares the terminal iterate of each method
	with the best terminal iterate across all methods.}}
	\label{fig:performance-fraction-problems-solved}
	\smallskip
	{\footnotesize
	\fbox{\begin{tabular}{@{}ccccc@{}}
	\textcolor[HTML]{000000}{\rule[0.5ex]{1.4em}{0.8pt}}~This paper &
	\textcolor[HTML]{0072B2}{\rule[0.5ex]{1.4em}{0.8pt}}~Conference version &
	\textcolor[HTML]{009E73}{\rule[0.5ex]{1.4em}{0.8pt}}~UTR &
	\textcolor[HTML]{E69F00}{\rule[0.5ex]{1.4em}{0.8pt}}~ARC &
	\textcolor[HTML]{CC79A7}{\rule[0.5ex]{1.4em}{0.8pt}}~TRU
	\end{tabular}}\par}
	\smallskip
	\scalebox{1}[1.2]{\includegraphics[width=0.495\textwidth]{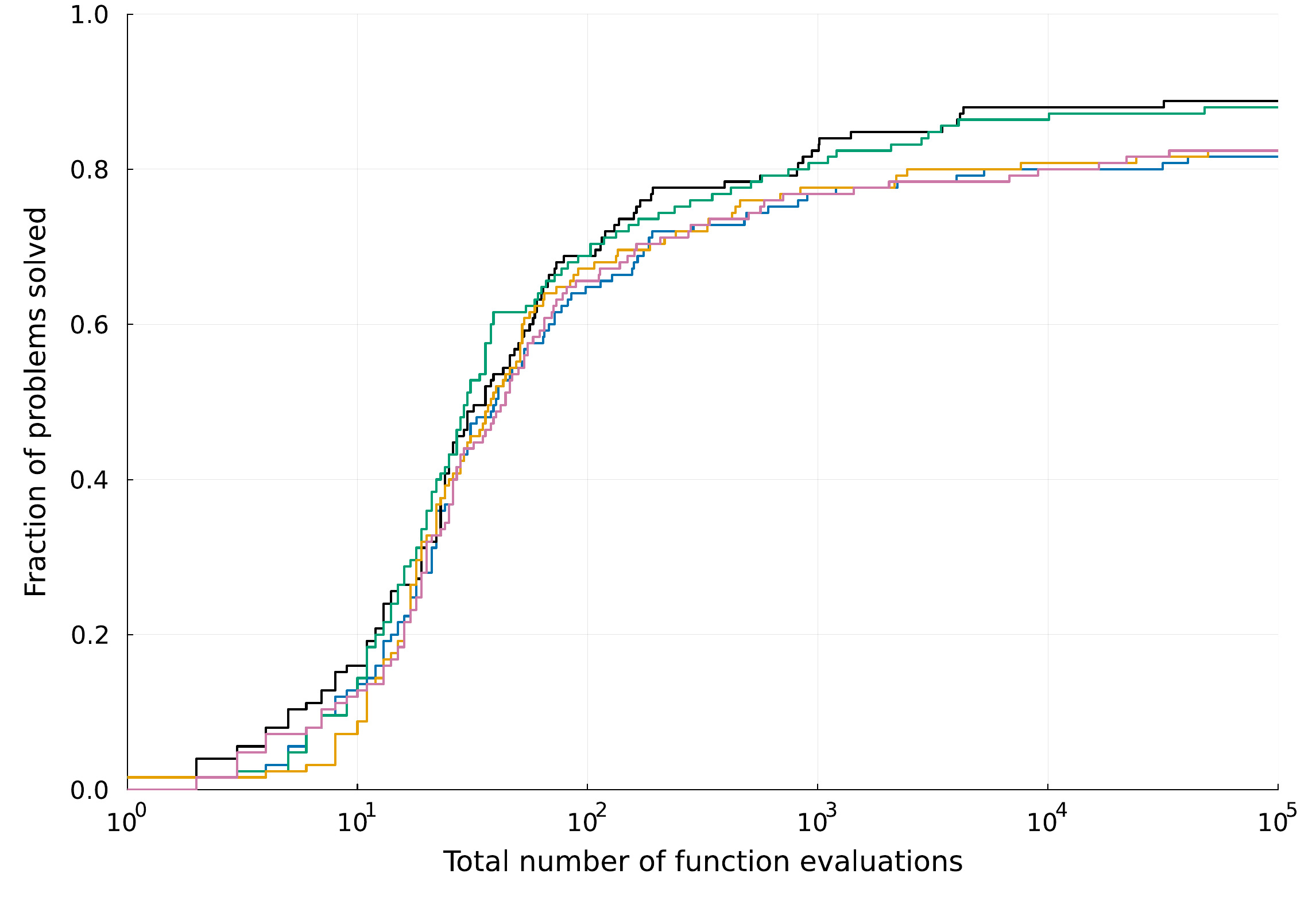}}
	\scalebox{1}[1.2]{\includegraphics[width=0.495\textwidth]{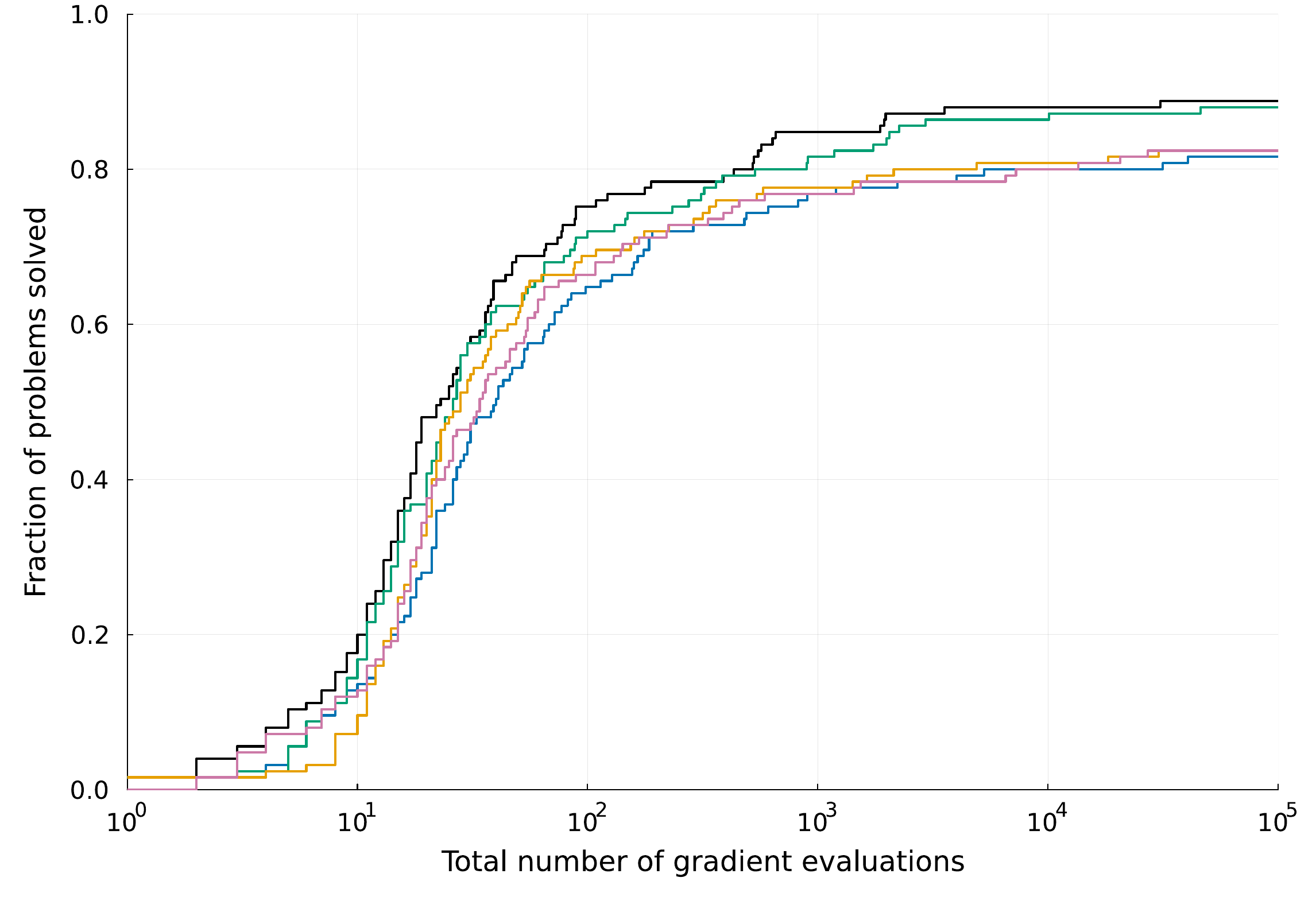}}
	\scalebox{1}[1.2]{\includegraphics[width=0.495\textwidth]{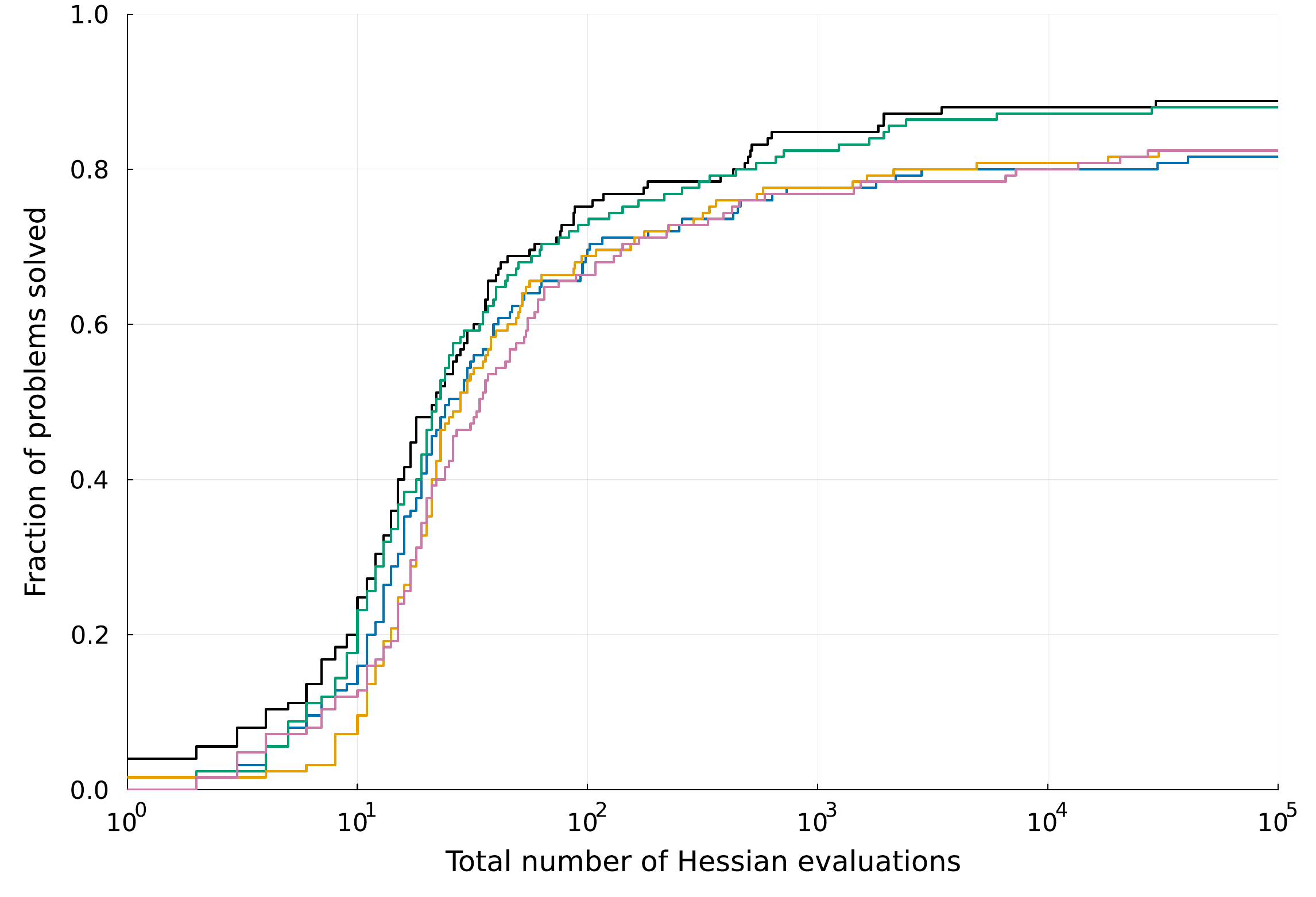}}
	\scalebox{1}[1.2]{\includegraphics[width=0.495\textwidth]{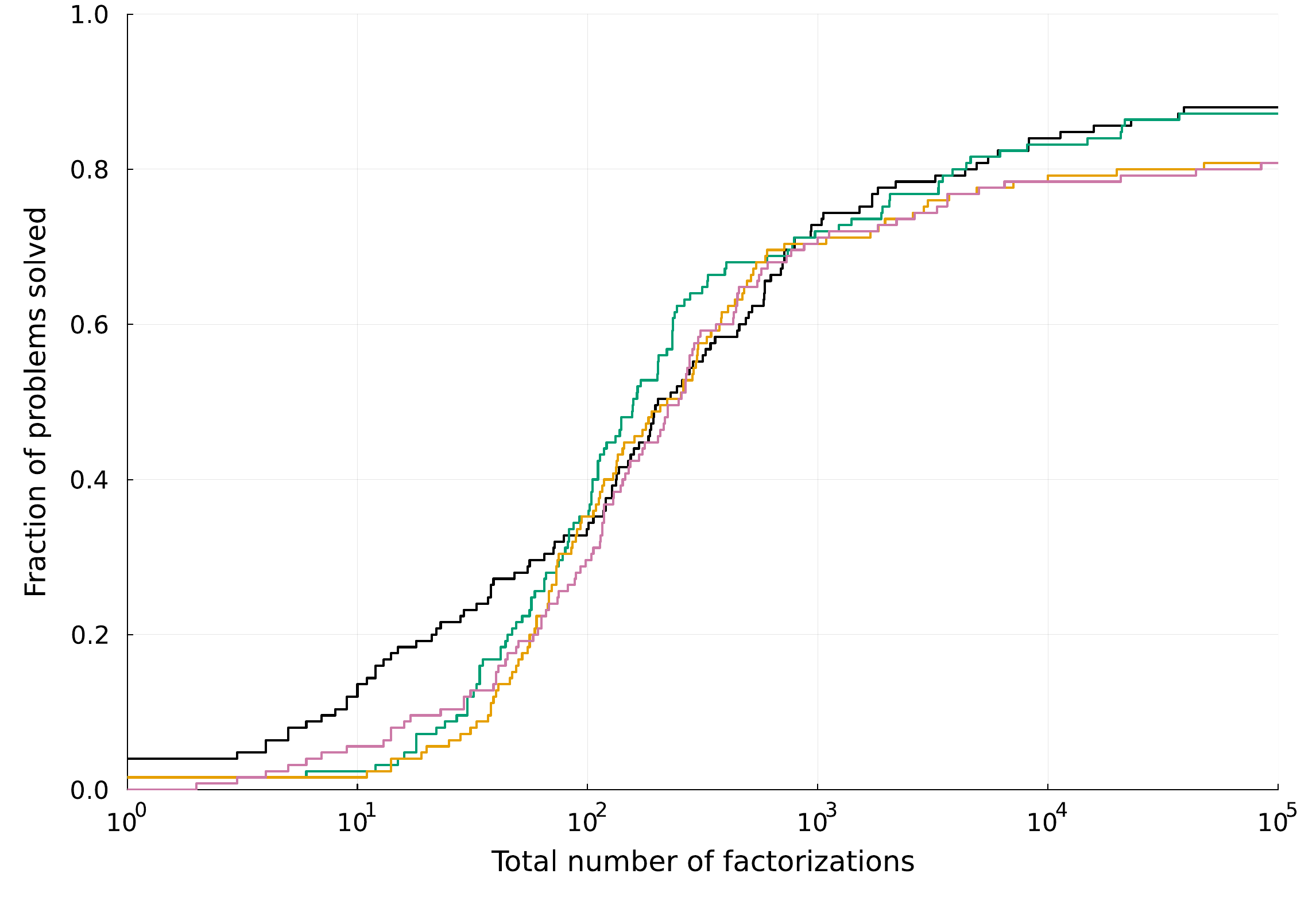}}
	\scalebox{1}[1.2]{\includegraphics[width=0.495\textwidth]{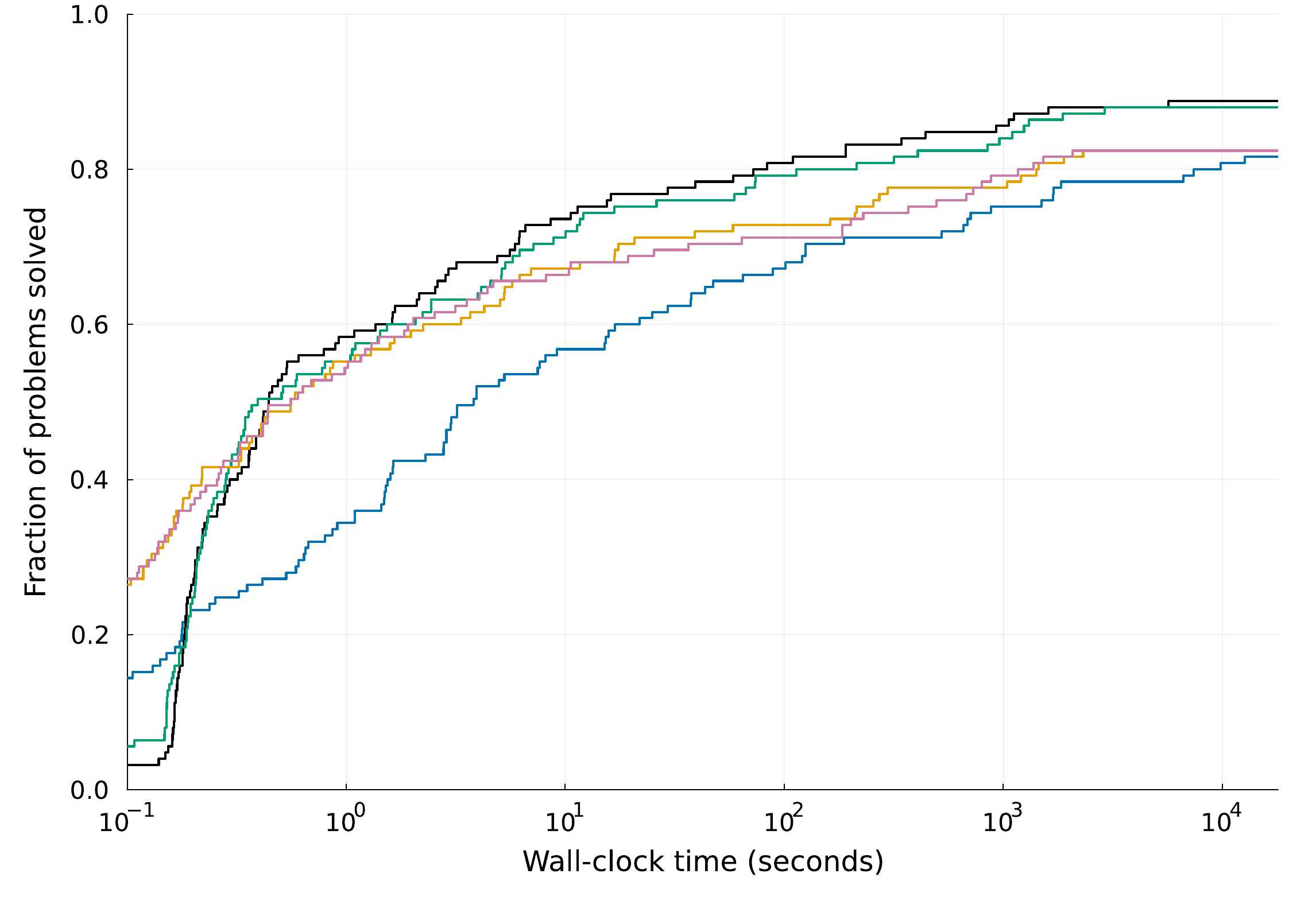}}
	\scalebox{1}[1.2]{\includegraphics[width=0.495\textwidth]{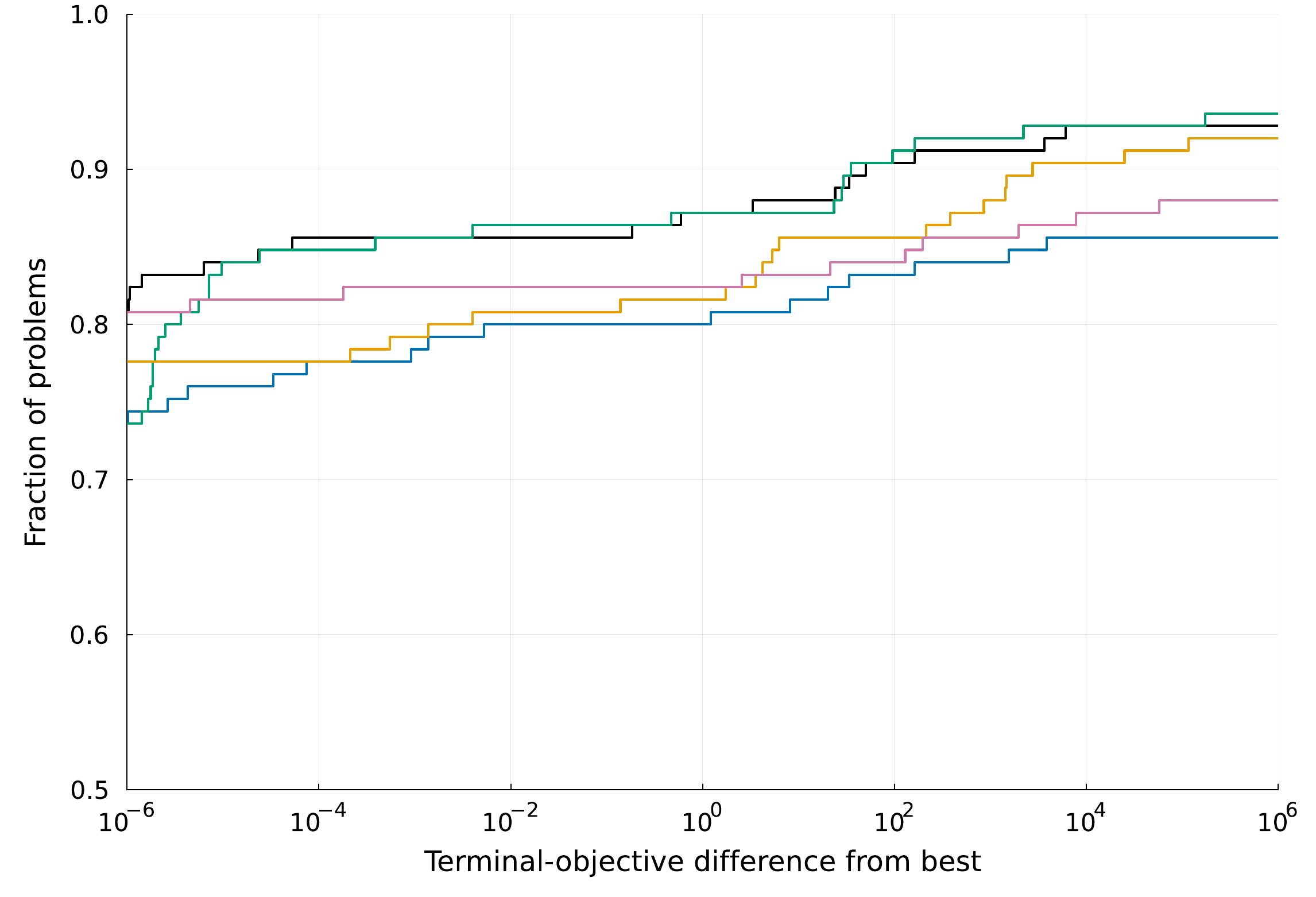}}
\end{figure}

\clearpage
\appendix

\section{\edit{Comparison with the conference version of the paper}}\label{sec:comparison-conference-version}
In this section, we will comment on the similarities and differences between the current version of our method and the conference version \citep{hamad2022consistently}. In Subsection~\ref{sec:ablation-study-results}, we will show how these modifications improved performance.

The most significant change is the approach for solving the trust-region subproblem. Previously, in the conference version, the termination criterion for the trust-region subproblem was $\|\grad \Mk(d_k) + \delta_k d_k \| \le \SPTone \| \grad f(x_k + d_k)\|$. This condition was impractical as it required  computing  the next gradient to validate the trust-region subproblem solution $d_k$ even when the step was not accepted, wasting gradient evaluations. To address this, the conference version was implemented with $\SPTone = 0$. Consequently, for the hard case in the subproblem solves, we computed the singular value decomposition (SVD) of the Hessian (see discussion of ``hard case'' in \citet[Chapter 4]{nocedal2006numerical} for more details). This approach was slow, as it did not take advantage of the sparsity of the Hessian and was susceptible to numerical errors in the singular value decomposition.

This paper resolves these issues by replacing $\|\grad \Mk(d_k) + \delta_k d_k \| \leq \SPTone \| \grad f(x_k + d_k)\|$ with \eqref{eq:model-gradient-weaker} which depends only on the gradient norm of previous iterates. Similarly, our termination condition now checks that $\varepsilon_k \leq \epsilon$ (Theorem~\ref{thm:main-fully-adaptive-result}) instead of $\|\grad f(x_k + d_k)\| \leq \epsilon$. While these changes may seem subtle, they forced
us to completely redesign the proofs.
\editTwo{In particular, unlike \citet[Theorem 1]{hamad2022consistently} it is no longer 
true that there is a universal threshold 
for each target tolerance $\epsilon > 0$
such that below this threshold the search 
direction size will increase. 
For this reason, we had to introduce the 
repeated phase argument (see \Cref{sec:main-result}) which is the 
main conceptual change in the proof. The repeated phase argument circumvents this issue
by defining two phases for each time $\varepsilon_j$ halves 
and aggregating the analysis across all the repeating phases.
In contrast, \citet[Theorem 1]{hamad2022consistently} only needs 
to analyze a single pair of phases (called periods) due to this simple threshold structure.}

For the hard case, we now use inverse-power iteration to compute \editTwo{an approximate minimum eigenvector} instead of an SVD, and then we generate the search direction as described in \Cref{sec:Solving the trust-region sub-problem}. This approach is significantly faster. Details of our revised trust-region subproblem solver appear in Appendix~\ref{sec:Solving the trust-region sub-problem}
and can be contrasted with the solver presented in the conference version \cite[Appendix C]{hamad2022consistently}. 

There are three additional, more minor modifications, aimed at improving the practical performance of the method:
\begin{itemize}
	\item We added a heuristic for selecting the initial radius: $r_1 = \frac{10 \| \grad f(x_1)\|}{\| \grad ^ 2 f(x_1) \|}$ instead of a fixed radius of $r_1 = 1$, as was done in the conference version \citep{hamad2022consistently}.
	The goal of this choice is to make the algorithm invariant to scaling. 
	Specifically, consider any function $f$ and starting point $\hat{x}_0$. Then, for all values of the scalar $\alpha \neq 0$ the algorithm applied to $f(\alpha x)$ starting from $x_0 = \hat{x}_0/\alpha$
	will have the same iterates, up to the scaling factor $\alpha$. In the corner case when the spectral norm of the Hessian at the starting iterate is zero, i.e., $\| \grad ^ 2 f(x_1) \| = 0$, we set $r_1 = 1$.
	\item We changed the radius update rule. In the conference version \citep[Algorithm~1]{hamad2022consistently}, when $\hat{\rho}_k < \beta_1$, we set $r_{k+1}$ to $\|d_k\| / \omega_1$; otherwise, we set $r_{k+1}$ to $\omega_1 \|d_k\|$. However, in Algorithm~\ref{alg:CAT} when $\hat{\rho}_k < \editTwo{\beta}$, we set $r_{k+1}$ to $r_k / \omega_1$; otherwise, we set $r_{k+1}$ to $\max\{\omega_2 \|d_k\|, r_k\}$.
	This change ensures that the radius does not shrink on successful
	steps even when the search direction is small.
	\item We modified the computation of $\hat{\rho}$. The definition of $\hat{\rho}$ given in \eqref{eq:rho-hat-k} replaces $\|\grad f(x_k + d_k)\|$ in \citet[Equation~(8)]{hamad2022consistently} with \edit{ the term }$\min\{\|\grad f(x_k)\|, \|\grad f(x_k + d_k)\|\}$. This change makes steps more likely to be considered successful.
\end{itemize}

\section{Ablation study}\label{sec:ablation-study-results}
To justify our algorithmic choices, we conducted an ablation study evaluating the consequences of disabling each feature of the algorithm separately. 
The modifications we study are: 
\begin{enumerate}
	\item \emph{Rho hat rule.} We use $\hat{\rho}$ defined in \eqref{eq:rho-hat-k} instead of $\rho$ defined in \eqref{eq:classic-rho-k}, as in classical trust-region methods.
	\item \emph{Radius update rule.} We update $r_{k+1}$ to be $\max\{\omega_2 \|d_k\|, r_k\}$ instead of $\omega \|d_k\|$, which is used in \citep{hamad2022consistently} when $\hat{\rho}_k \geq \beta$. Otherwise, we update $r_{k+1}$ to be $r_k / \omega_1$ instead of $\|d_k\| / \omega$, which is used in \citep{hamad2022consistently}. We use $\omega_2 = 16.0$ and $\omega_1 = 8.0$. In the conference version of this work \citep{hamad2022consistently}, $\omega$ has the same value as $\omega_1$.
	\item \editTwo{\emph{Set $b_k=0$}. Instead of using $\GainThres_k =  \defaultAlgConst \varepsilon_{k} \| d_{k} \| + 10^{-8} (\abs{f(x_k)} + 1 )$ as per \Cref{sec:our-trust-region-subproblem} we set $b_k=0$.}
	\item \emph{Initial radius rule.} We use $r_1 = 10 \frac{\|\grad f(x_1)\|}{\|\grad ^ 2 f(x_1)\|}$ instead of the default starting radius $r_1 = 1.0$ that we used in the conference version of this work \citep{hamad2022consistently}.
	\item \emph{New trust-region subproblem solver.} We use the new trust-region subproblem solver (Appendix~\ref{sec:Solving the trust-region sub-problem}) instead of the old trust-region subproblem solver developed in the conference version of this work \citep[Appendix C]{hamad2022consistently}.
\end{enumerate} 

\editTwo{The experiment was setup exactly as per \Cref{table:cutest-run-statistics-combined} of Section~\ref{sec:numerical-results}.
In particular,
\Cref{fig:ablation-study} plots the shifted 
geometric mean of various metrics with each 
algorithmic feature separately disabled.

The first three features are needed 
to maintain our theoretical guarantees (in particular, they ensure new trust-region subproblem termination criteria provides an algorithm with our desired guarantees), and 
generally have minimal impact on performance. 
The only places where disabling these changes improves performance is when we disable the conference radius update: the function evaluations drop from $106.4$ to $99.1$ and the runtime drops from $8.87$ seconds to $8.38$ seconds,
or when we set $b_k=0$ the gradient evaluations 
drop from $81.1$ to $78.3$. 
Otherwise these theoretical algorithmic features either improve performance or make no difference.

The fourth and fifth features are purely practical, the theory still works if they are enabled or disabled. However, the new trust-region subproblem solver could not be applied to the conference paper due to the lack of an implementable subproblem termination criteria.
From these figures we can see that while the radius update does give consistent improvement the most important change is the new trust-region subproblem solver.
}

\begin{figure}[htbp]
	\centering
	\caption{Shifted geometric means (shift 1.0) for all ablations. Red bars are changes for which the theory in this paper no longer applies; blue bars change only practical implementation choices. The dashed line gives the value for the method in this paper.}
	\label{fig:ablation-study}
	\smallskip
	{\small
	\fbox{\begin{tabular}{@{}c@{}}
	\raisebox{0.5ex}{%
		\rule{1.15em}{0.9pt}\hspace{0.35em}%
		\rule{1.15em}{0.9pt}\hspace{0.35em}%
		\rule{1.15em}{0.9pt}}~This paper
	\qquad
	\textcolor[HTML]{C00000}{\rule{1.2em}{0.7em}}~Changes that break the theory
	\end{tabular}}\par}
	\smallskip
	\makebox[\textwidth][c]{%
		\includegraphics[height=0.27\textwidth]{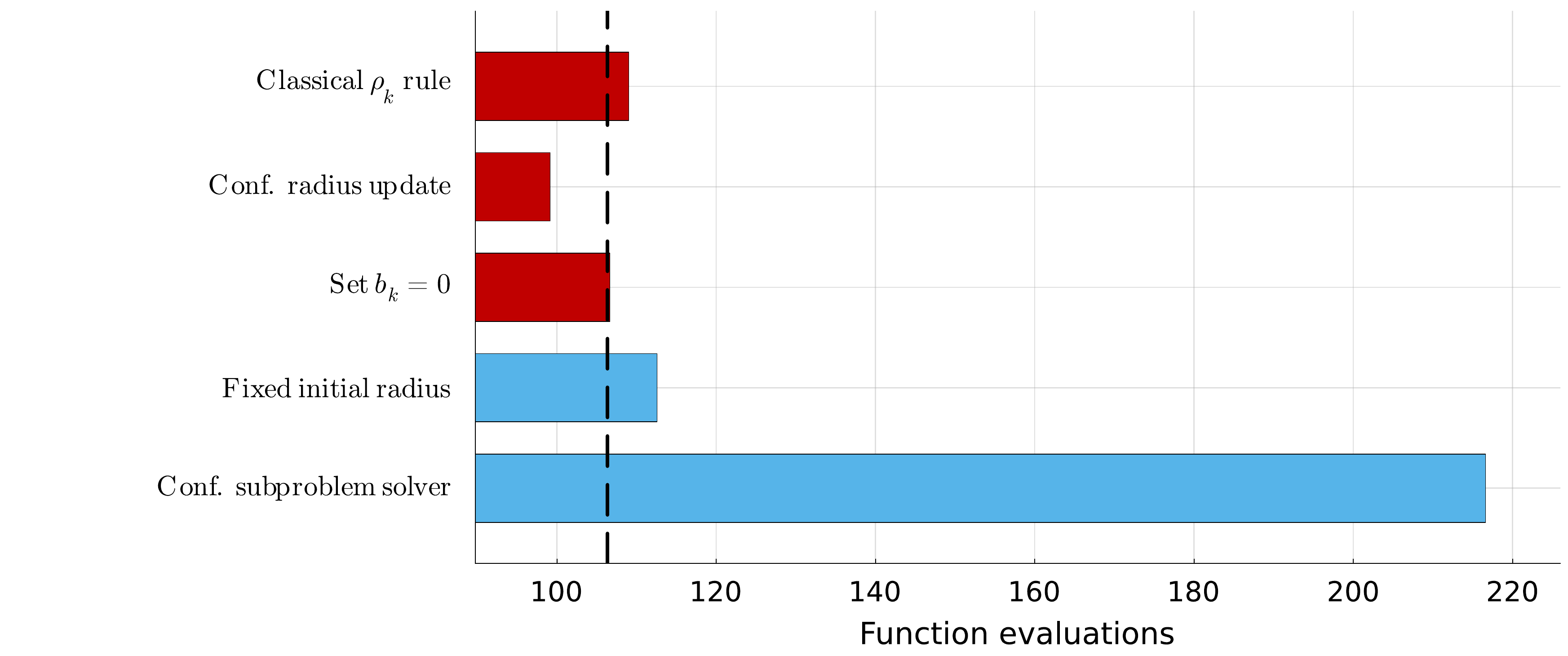}%
		\includegraphics[height=0.27\textwidth]{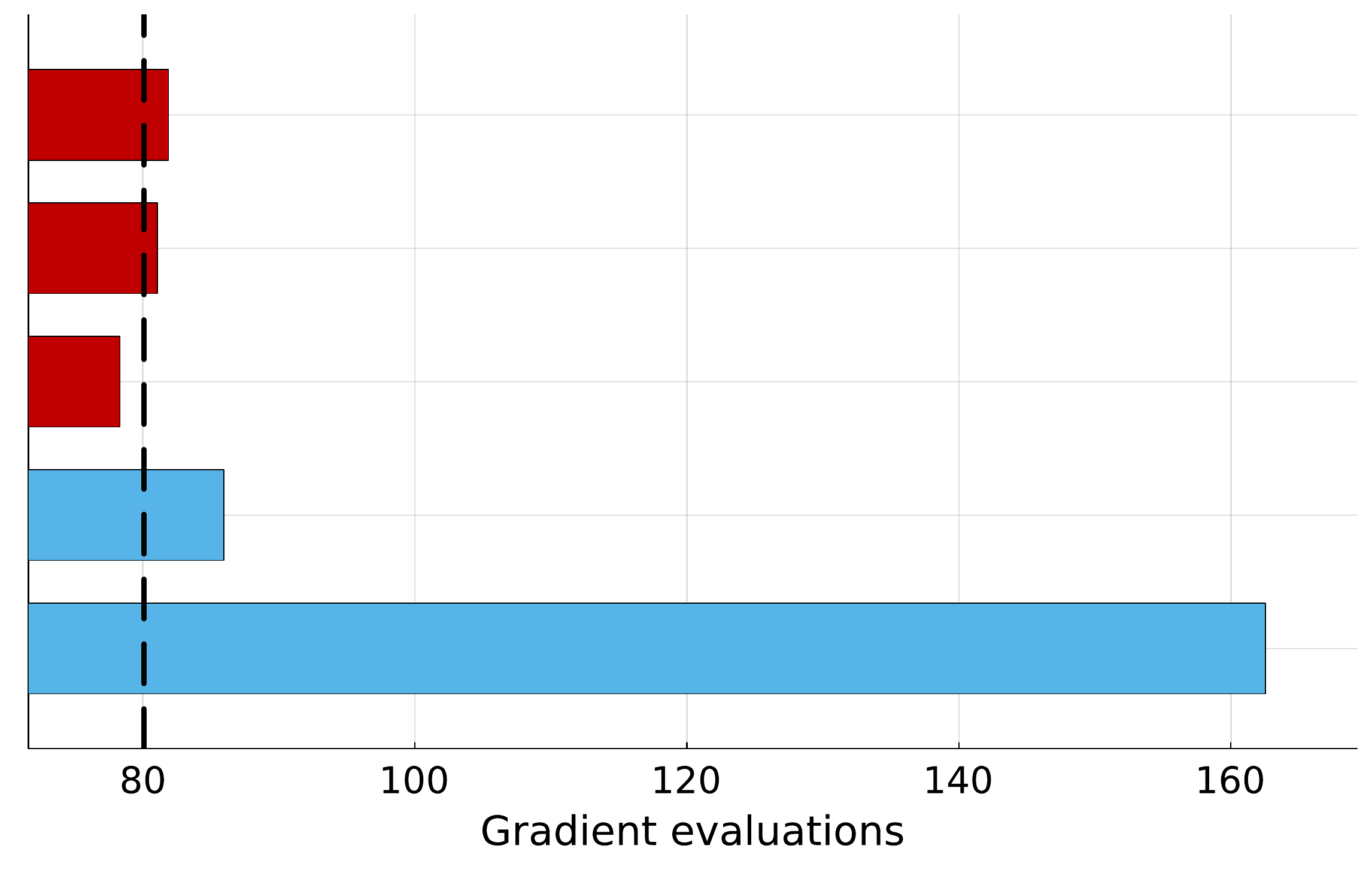}%
	}\par
	\makebox[\textwidth][c]{%
		\includegraphics[height=0.27\textwidth]{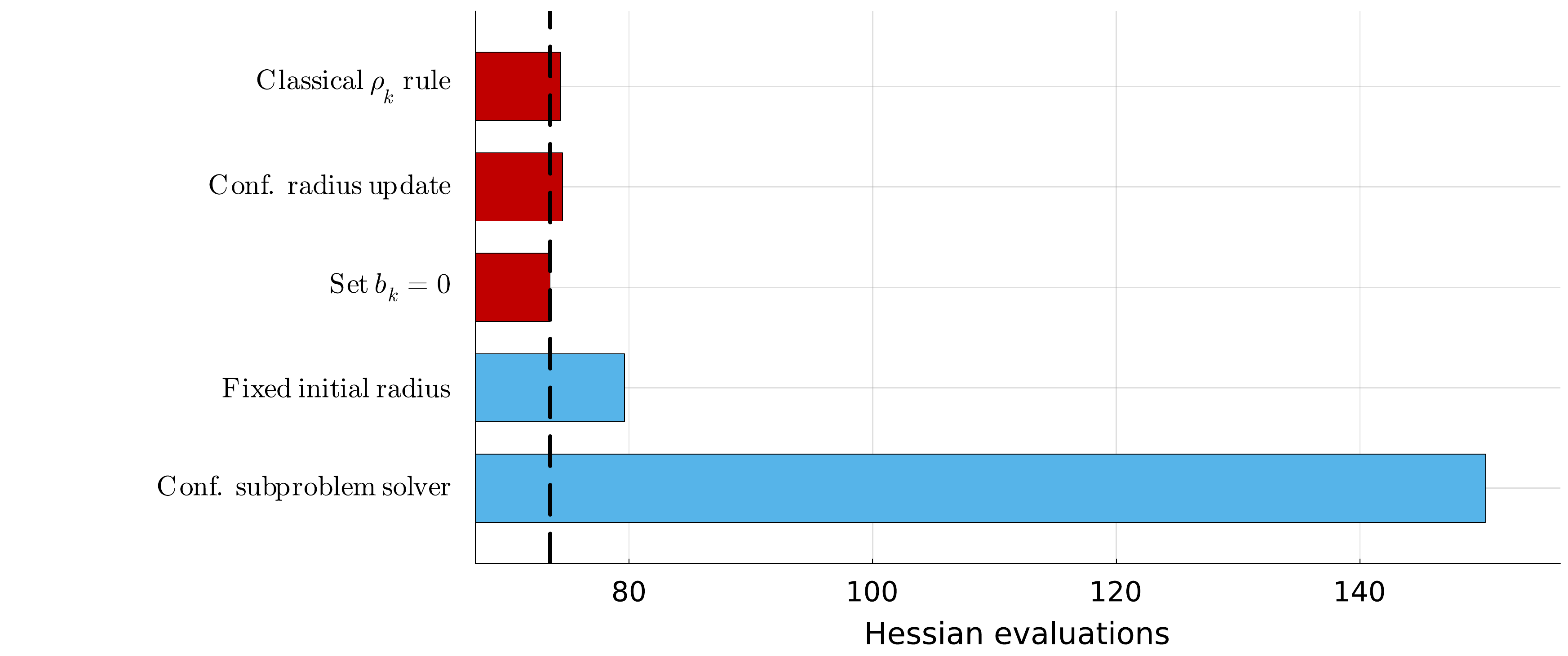}%
		\includegraphics[height=0.27\textwidth]{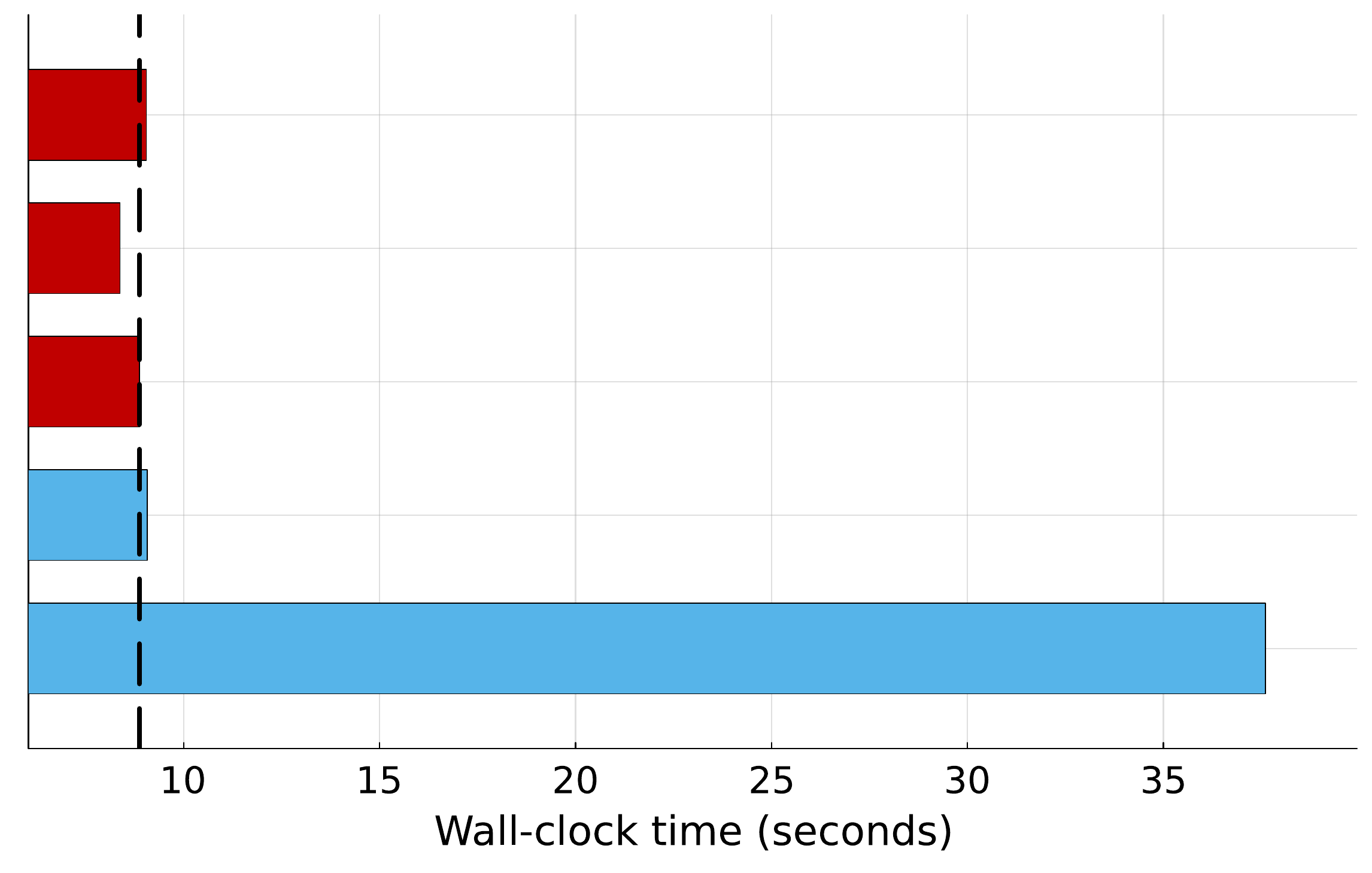}%
	}
\end{figure}

\section{Solving the trust-region subproblem}\label{sec:Solving the trust-region sub-problem}
In this section, we detail our approach to satisfy the trust-region subproblem termination criteria given in \eqref{eq:subproblem-termination-criteria} \editTwo{for $\SPTone, \SPTtwo, \SPTthree \in (0,1)$}. We first attempt to take a Newton step by checking if $\grad^2 f(x_k) \succ 0$ and $\|\grad^2 f(x_k)^{-1} \grad f(x_k)\| \leq r_k$. 
If this condition is not satisfied, we employ Algorithm~\ref{alg:TRS-New}. The basis of Algorithm~\ref{alg:TRS-New} is a univariate function $\phi$, whose root \editTwo{when paired with the regularization parameter defined below} solves the subproblem termination criteria.  The function $\phi$ is defined as:
\[
\phi(\delta) := \begin{cases}
	\editTwo{-1} & \text{if~} \grad^2 f(x_k) + \delta \eye \nsucc 0 \text{ or } \|d_k^{\delta} \| >  r_k \\
	0 & \text{if~} \SPTtwo r_k \leq \|d_k^{\delta}\|  \leq r_k \\
	0 & \text{if~} \|d_k^{\delta}\|  \leq r_k \And  \|\grad{\Mk(d_k^{\delta})}\| \leq \SPTone \varepsilon_k \\ 
	\editTwo{+1} & \editTwo{\text{otherwise}}
\end{cases}
\]
where $d_k^{\delta} := -(\grad^2 f(x_k) + \delta \eye) ^ {-1} \grad f(x_k)$. \editTwo{In practice, the linear system is not exactly solved and we apply up to three 
iterations of iterative refinement until:
\begin{flalign}\label{eq:iterative-refinement-tolerance}
\| \grad \Mk(d_k^{\delta}) + \delta d_k^{\delta} \| \le \SPTone \varepsilon_k \min\left\{
\frac{1}{3},
\frac{1-\SPTthree}{1+\SPTthree}
\right\} \quad\&\quad \Mk(d_k^{\delta}) \le 0.
\end{flalign}
This
condition is designed to ensure it is possible to satsify the subproblem termination criteria given in \Cref{eq:subproblem-termination-criteria}. If \Cref{eq:iterative-refinement-tolerance} does not hold after iterative refinement, the algorithm
continues ignoring this violation.
However, the final direction produced by the algorithm is still validated to check satisfies
\Cref{eq:subproblem-termination-criteria}.}

We first find an initial interval where $\phi$ changes sign, and then apply bisection to this interval.
The bisection routine stops either when the trust-region subproblem termination criteria is satisfied, or when the interval size becomes sufficiently small, i.e., 
\begin{subequations}\label{eq:delta-interval-hard-case}
\begin{flalign}
\delta_{\editTwo{+}} - \delta_{\editTwo{-}} &\le \frac{\SPTone \varepsilon_k}{r_k} \editTwo{\min\left\{\frac{1}{6},\frac{1-\SPTthree}{3}\right\}} \\
\editTwo{\grad^2 f(x_k) + \delta_{-} \eye} &~ \editTwo{\nsucc 0} \\
\editTwo{\phi(\delta_{+})} &\editTwo{~= 1}.
\end{flalign}
\end{subequations}
In the latter case, the trust-region subproblem is marked as a hard case and we use inverse-power iteration to find an approximate minimum eigenvector of the Hessian, $y$. 
The trust region subproblem search direction becomes 
\begin{flalign}\label{eq:hard-case-search-direction}
d_k = d_k ^ {\editTwo{\delta_{+}}} + \alpha y
\end{flalign}
where the scalar $\alpha$ minimizes \editTwo{$\Mk(d_k)$ subject to $\|d_k\| = r_k$.}

If, due to numerical errors, the inverse-power iteration approach fails, we attempt to compute a search direction with a perturbed gradient $\grad{f(x_k)} + 0.5 \SPTone \varepsilon_k u_k $, where $u_k$ is a random unit vector. \editTwo{The backup uses $\SPTone/2$ instead of $\SPTone$ for its tolerance on \Cref{eq:model-gradient-weaker} and the returned direction is checked against the original subproblem termination criteria.}

If we fail to generate a search direction that \editTwo{satisfies} \Cref{eq:subproblem-termination-criteria},  Algorithm~\ref{alg:CAT} terminates with a failure status code  `TRUST\_REGION\_SUBPROBLEM\_ERROR'.
These failures are reported in Table~\ref{label:failures} as numerical errors.

\begin{algorithm}[p]
	\small
	\caption{Trust-region subproblem solver (assuming Newton step fails)}
	\label{alg:TRS-New}
	\begin{algorithmic}[1] 
		\STATE \textbf{Input:} $\grad^2 f(x_k)$, \editTwo{$\grad f(x_k)$, }$\varepsilon_k$, $\delta_{k-1}$, $r_k$, $\SPTone$
		\STATE \textbf{Output:} Search direction \editTwo{$d_k$ and associated regularization parameter $\delta_k$} 

		\STATE $[\editTwo{\delta_{-}}, \editTwo{\delta_{+}}] \gets \textsc{FindInitialInterval}(\delta_{k-1})$
		\STATE $\editTwo{\delta_{-}}, \editTwo{\bar{\delta}}, \editTwo{\delta_{+}}, d_k^{\editTwo{\bar{\delta}}}, d_k^{\editTwo{\delta_{+}}} \gets \textsc{Bisection}(\grad^2 f(x_k), \varepsilon_k, r_k, \editTwo{\delta_{-}}, \editTwo{\delta_{+}}, \SPTone)$
		
		\IF{$\phi(\editTwo{\bar{\delta}})=0$}
			\STATE \textbf{return} $d_k^{\editTwo{\bar{\delta}}}, \editTwo{\editTwo{\bar{\delta}}}$
		\ELSE
			\STATE \textbf{return} $\textsc{InversePowerIteration}(
				\grad^2 f(x_k), \varepsilon_k, r_k, \editTwo{\delta_{+}},
				d_k^{\editTwo{\delta_{+}}}, \SPTone)$
		\ENDIF

		\vspace{0.5em}
		\STATE \textbf{Function} \textsc{FindInitialInterval}($\delta_{k-1}$)
		\IF{$\phi(\delta_{k-1}) = 0$}
		\STATE \textbf{return} $[\delta_{k-1}, \delta_{k-1}]$
		\ENDIF
		\IF{$\delta_{k-1} = 0$ \editTwo{ or  $k=0$}}
		\STATE $\delta_{k-1} \gets 1$
		\ENDIF
		\FOR{$i = 1$ to $\infty$}
		\STATE $x \gets
		\begin{cases}
			\delta_{k-1} & \text{if } i = 1 \\
			(2^{(i-1)^2})^{\editTwo{-}\phi(\delta_{k-1})} \cdot \delta_{k-1} & \text{otherwise}
		\end{cases}$
		\STATE $y \gets \delta_{k-1} \cdot (2^{i^2})^{\editTwo{-}\phi(\delta_{k-1})}$
		\IF{$\phi(x)=0$} \STATE \textbf{return} $[x,x]$ \ELSIF{$\phi(y)=0$} \STATE \textbf{return} $[y,y]$ \ELSIF{$\phi(x)\phi(y)<0$} \STATE \textbf{return} $[\min\{x,y\},\max\{x,y\}]$ \ENDIF
		\ENDFOR
		
		\vspace{0.5em}
		\STATE \textbf{Function} \textsc{Bisection}($\grad^2 f(x_k)$, $\varepsilon_k$, $r_k$, $\editTwo{\delta_{-}}$, $\editTwo{\delta_{+}}$, $\SPTone$)
		\FOR{$i = 1$ to $\infty$}
		\STATE $\editTwo{\bar{\delta}} \gets (\editTwo{\delta_{-}} + \editTwo{\delta_{+}})/2$
		\IF{$\phi(\editTwo{\bar{\delta}}) = 0$}
		\STATE \textbf{return} $\editTwo{\delta_{-}}, \editTwo{\bar{\delta}}, \editTwo{\delta_{+}}, d_k^{\editTwo{\bar{\delta}}}, d_k^{\editTwo{\delta_{+}}}$
		\ELSIF{$\phi(\editTwo{\bar{\delta}}) = \editTwo{-1}$}
		\STATE $\editTwo{\delta_{-}} \gets \editTwo{\bar{\delta}}$
		\ELSE
		\STATE $\editTwo{\delta_{+}} \gets \editTwo{\bar{\delta}}$
		\ENDIF
		\IF{\editTwo{\Cref{eq:delta-interval-hard-case} holds}}\label{line:switch-to-hard-case}
		\STATE \textbf{return} $\editTwo{\delta_{-}}, \editTwo{\delta_{+}}, \editTwo{\delta_{+}}, d_k^{\editTwo{\delta_{+}}}, d_k^{\editTwo{\delta_{+}}}$ \COMMENT{Hard case}
		\ENDIF
		\ENDFOR
		
		\vspace{1em}
		\STATE \textbf{Function} \textsc{InversePowerIteration}($\grad^2 f(x_k)$, $\varepsilon_k$, $r_k$, $\editTwo{\delta_{+}}$, $d_k^{\editTwo{\delta_{+}}}$, $\SPTone$)
		\STATE $y \sim \text{NormalDistribution}(0,1)^n$
		\STATE \editTwo{$y \gets y/\|y\|$}
		\FOR{$i = 1$ to $\infty$}
		\STATE \editTwo{$y \gets (\grad^2 f(x_k) + \editTwo{\delta_{+}} \eye)^{-1} y$}
		\STATE \editTwo{$y \gets y / \|y\|$}
		\STATE Find $\alpha$ that \editTwo{minimizes $M_k(d_k^{\editTwo{\delta_{+}}}+\alpha y)$ subject to $\|d_k^{\editTwo{\delta_{+}}} + \alpha y\| = r_k$}
		\IF{\eqref{eq:subproblem-termination-criteria} is satisfied \editTwo{by $(d_k^{\editTwo{\delta_{+}}}+\alpha y,\editTwo{\delta_{+}})$}}
		\STATE \textbf{return} $d_k^{\editTwo{\delta_{+}}} + \alpha y$\editTwo{, $\editTwo{\delta_{+}}$}
		\ENDIF
		\ENDFOR
		
	\end{algorithmic}
\end{algorithm}

\section*{Acknowledgments}
The authors were supported by AFOSR grant \#FA9550-23-1-0242. The authors would like to thank Coralia Cartis and Nicholas Gould for their helpful feedback on a draft of this paper, and Nicholas Gould for his suggestions on running the GALAHAD package. The authors would also like to thank Akif Khan for identifying errors in an early draft and for the helpful discussions.

\clearpage

\bibliographystyle{unsrtnat}
\bibliography{./references}
\end{document}

%% file: CAT_shared.tex
\usepackage{graphicx}
\usepackage{algorithmic}
\usepackage{multirow}
\usepackage{ifpdf}

\ifpdf
\DeclareGraphicsExtensions{.eps,.pdf,.png,.jpg}
\else
\DeclareGraphicsExtensions{.eps}
\fi

\ifSIAM
\newsiamremark{remark}{Remark}
\newsiamremark{hypothesis}{Hypothesis}
\crefname{hypothesis}{Hypothesis}{Hypotheses}
\newsiamthm{claim}{Claim}
\newsiamremark{fact}{Fact}
\crefname{fact}{Fact}{Facts}
\fi

\ifARXIV
\usepackage{xcolor}    
\usepackage{hyperref}  
\usepackage{cleveref}  
\usepackage{algorithm} 

\newenvironment{keywords}{\vspace{0.5em}\noindent\textbf{Key words.}\quad}{\vspace{0.5em}}
\newenvironment{MSCcodes}{\noindent\textbf{AMS subject classifications.}\quad}{\vspace{0.5em}}

\newtheorem{theorem}{Theorem}
\newtheorem{lemma}[theorem]{Lemma}
\newtheorem{claim}[theorem]{Claim}
\newtheorem{fact}[theorem]{Fact}

\fi

\ifSIAM
\headers{A simple and practical adaptive trust-region method}{Fadi Hamad and Oliver Hinder}
\fi

\ifARXIV
\newcommand{\funding}[1]{#1}
\fi

\title{A simple and practical adaptive trust-region method\thanks{This work is an improvement of our conference paper that has been published at NeurIPS 2022.
		\funding{This work was funded by AFOSR grant \#FA9550-23-1-0242}}}

\ifSIAM
\author{Fadi Hamad\thanks{Department of Industrial Engineering, University of Pittsburgh, 
		\email{fah33@pitt.edu}.}
	\and Oliver Hinder\thanks{Department of Industrial Engineering, University of Pittsburgh,
		\email{ohinder@pitt.edu}.}}
\fi

\usepackage[numbers,square,compress,sort]{natbib}
\usepackage{tabularx}
\usepackage{physics}
\usepackage{booktabs}
\usepackage{amssymb}
\usepackage{amsmath}

\usepackage{accents}
\newcommand{\ubar}[1]{\underaccent{\bar}{#1}}
\usepackage{enumitem} 

\newcommand{\edit}[1]{{#1}}
\ifARXIV
\newcommand{\editTwo}[1]{{#1}}
\fi
\ifSIAM
\newcommand{\editTwo}[1]{{\color{blue}#1}}
\fi

\newcommand{\possibleImprovement}[1]{{#1}}

\newcommand{\eye}{\mbox{\bf I}}

\newcommand{\argmin}{\mathop{\rm argmin}}

\def\N{\mathbb{N}}

\def\R{\mathbb{R}}

\newcommand{\dminInd}[1]{\ubar{d}_{#1}}
\newcommand{\dmaxInd}[1]{\bar{d}_{#1}}

\newcommand{\GainThres}{b}
\newcommand{\AlgConst}{\xi}
\newcommand{\TheoryAlgConst}{C}
\newcommand{\defaultAlgConst}{0.1}

\newcommand{\undereq}[1]{\underset{#1}{=}}
\newcommand{\underle}[1]{\underset{#1}{\le}}
\newcommand{\underless}[1]{\underset{#1}{<}}
\newcommand{\undergreater}[1]{\underset{#1}{>}}
\newcommand{\underge}[1]{\underset{#1}{\ge}}

\newcommand{\SPTone}{\gamma_1}
\newcommand{\SPTtwo}{\gamma_2}
\newcommand{\SPTthree}{\gamma_3}

\newcommand{\RequireSPT}{$\SPTtwo \in (1/\omega_1, 1]$, $\SPTthree \in (0,1]$}

\newtheorem{assumption}{Assumption}

%% file: table-content/main/cutest_performance_statistics.tex
  \begin{tabular}{>{\raggedright\arraybackslash}p{2.6cm}lrrrrr}
    \toprule
    metric & method & \#$f$ & \#$\nabla f$ & \#$\nabla^2 f$ & \#fact. & time (seconds) \\
    \midrule
    \multirow{5}{2.6cm}{shifted geometric mean (shift value of 1.0)} & This paper & \textbf{106.4} & \textbf{80.1} & \textbf{73.5} & \textbf{360.1} & \textbf{8.86} \\
     & Conf. version \citep{hamad2022consistently} & 203.3 & 203.3 & 155.9 & N/A & 37.39 \\
     & UTR \citep{jiang2023universal} & 117.5 & 104.0 & 87.7 & 434.4 & 10.54 \\
     & ARC \citep{gould2003galahad} & 192.3 & 167.4 & 167.4 & 647.9 & 17.18 \\
     & TRU \citep{gould2003galahad} & 198.8 & 177.5 & 177.5 & 650.4 & 18.02 \\
    \bottomrule
  \end{tabular}

%% file: table-content/main/cutest_failure_reasons.tex
  \begin{tabular}{lrrrrr}
    \toprule
    Method & Iter. limit & Max time & Small step & Numerical error & Total \\
    \midrule
    This paper & 2 & 6 & 2 & 4 & \textbf{14} \\
    Conf. version \citep{hamad2022consistently} & 2 & 14 & 0 & 7 & 23 \\
    UTR \citep{jiang2023universal} & 4 & 5 & 6 & 0 & 15 \\
    ARC \citep{gould2003galahad} & 8 & 6 & 2 & 6 & 22 \\
    TRU \citep{gould2003galahad} & 8 & 9 & 3 & 2 & 22 \\
    \bottomrule
  \end{tabular}

%% file: references.bib
@article{powell2010convergence,
	title={On the convergence of a wide range of trust region methods for unconstrained optimization},
	author={Powell, Michael JD},
	journal={IMA journal of numerical analysis},
	volume={30},
	number={1},
	pages={289--301},
	year={2010},
	publisher={OUP}
}

@article{he2023newton,
	title={{N}ewton-{CG} methods for nonconvex unconstrained optimization with Holder continuous Hessian},
	author={He, Chuan and Lu, Zhaosong},
	journal={arXiv preprint arXiv:2311.13094},
	year={2023}
}

@techreport{newtons_method_slow_convergence_global_lip,
	author = {Cartis, Coralia and Gould, Nicholas IM and Toint, Philippe L},
	title = {An example of slow convergence for {N}ewton method on a function with globally {L}ipschitz continuous {H}essian},
	institution = {Namur Center for Complex Systems},
	year = {2011}
}

@article{leconte2024interior,
	title={An interior-point trust-region method for nonsmooth regularized bound-constrained optimization},
	author={Leconte, Geoffroy and Orban, Dominique},
	journal={arXiv preprint arXiv:2402.18423},
	year={2024}
}

@article{leconte2023complexity,
	title={Complexity of trust-region methods with unbounded {H}essian approximations for smooth and nonsmooth optimization},
	author={Leconte, Geoffroy and Habiboullah, Mohamed L and Orban, Dominique},
	journal={Mathematical Programming},
	year={2026}
}

@article{curtis2023worst,
	title={Worst-Case Complexity of {TRACE} with Inexact Subproblem Solutions for Nonconvex Smooth Optimization},
	author={Curtis, Frank E and Wang, Qi},
	journal={SIAM Journal on Optimization},
	volume={33},
	number={3},
	pages={2191--2221},
	year={2023},
	publisher={SIAM}
}

@article{jiang2023universal,
	title={Beyond Nonconvexity: A Universal Trust-Region Method with New Analyses},
	author={Jiang, Yuntian and He, Chang and Zhang, Chuwen and Ge, Dongdong and Jiang, Bo and Ye, Yinyu},
	journal={Journal of Scientific Computing},
	year={2025}
}

@article{gratton2024yet,
	title={Yet another fast variant of {N}ewton’s method for nonconvex optimization},
	author={Gratton, Serge and Jerad, Sadok and Toint, Philippe L},
	journal={IMA Journal of Numerical Analysis},
	year={2024},
	publisher={Oxford University Press}
}

@article{liu2022newton,
	title={A {N}ewton-{MR} algorithm with complexity guarantees for nonconvex smooth unconstrained optimization},
	author={Liu, Yang and Roosta, Fred},
	journal={arXiv preprint arXiv:2208.07095},
	year={2022}
}

@article{gould2003galahad,
	title={{GALAHAD}, a library of thread-safe Fortran 90 packages for large-scale nonlinear optimization},
	author={Gould, Nicholas IM and Orban, Dominique and Toint, Philippe L},
	journal={ACM Transactions on Mathematical Software (TOMS)},
	volume={29},
	number={4},
	pages={353--372},
	year={2003},
	publisher={ACM New York, NY, USA}
}

@article{gould2010solving,
	title={On solving trust-region and other regularised subproblems in optimization},
	author={Gould, Nicholas IM and Robinson, Daniel P and Thorne, H Sue},
	journal={Mathematical Programming Computation},
	volume={2},
	number={1},
	pages={21--57},
	year={2010},
	publisher={Springer}
}

@article{hamad2022consistently,
	title={A consistently adaptive trust-region method},
	author={Hamad, Fadi and Hinder, Oliver},
	journal={Advances in Neural Information Processing Systems},
	volume={35},
	pages={6640--6653},
	year={2022}
}

@article{nesterov2006cubic,
	title={Cubic regularization of {N}ewton method and its global performance},
	author={Nesterov, Yurii and Polyak, Boris T},
	journal={Mathematical Programming},
	volume={108},
	number={1},
	pages={177--205},
	year={2006},
	publisher={Springer}
}

@book{nocedal2006numerical,
	title={Numerical optimization},
	author={Nocedal, Jorge and Wright, Stephen},
	year={2006},
	publisher={Springer Science \& Business Media},
	address={New York}
}

@article{curtis2017trust,
	title={A trust region algorithm with a worst-case iteration complexity of $\mathcal{O}(\epsilon ^ {-3/2}) $ for nonconvex optimization},
	author={Curtis, Frank E and Robinson, Daniel P and Samadi, Mohammadreza},
	journal={Mathematical Programming},
	volume={162},
	number={1-2},
	pages={1--32},
	year={2017},
	publisher={Springer}
}

@article{cartis2011adaptiveI,
	title={Adaptive cubic regularisation methods for unconstrained optimization. {P}art {I}: motivation, convergence and numerical results},
	author={Cartis, Coralia and Gould, Nicholas IM and Toint, Philippe L},
	journal={Mathematical Programming},
	volume={127},
	number={2},
	pages={245--295},
	year={2011},
	publisher={Springer}
}

@article{cartis2011adaptiveII,
	title={Adaptive cubic regularisation methods for unconstrained optimization. {P}art {II}: worst-case function-and derivative-evaluation complexity},
	author={Cartis, Coralia and Gould, Nicholas IM and Toint, Philippe L},
	journal={Mathematical programming},
	volume={130},
	number={2},
	pages={295--319},
	year={2011},
	publisher={Springer}
}

@article{cartis2010complexity,
	title={On the complexity of steepest descent, {N}ewton's and regularized {N}ewton's methods for nonconvex unconstrained optimization problems},
	author={Cartis, Coralia and Gould, Nicholas IM and Toint, Ph L},
	journal={{SIAM} journal on optimization},
	volume={20},
	number={6},
	pages={2833--2852},
	year={2010},
	publisher={SIAM}
}

@article{gould2015cutest,
	title={{CUTEst}: a constrained and unconstrained testing environment with safe threads for mathematical optimization},
	author={Gould, Nicholas IM and Orban, Dominique and Toint, Philippe L},
	journal={Computational optimization and applications},
	volume={60},
	number={3},
	pages={545--557},
	year={2015},
	publisher={Springer}
}

@Misc{orban-siqueira-cutest-2020,
	author = {Orban, Dominique and Abel S Siqueira and {contributors}},
	title = {{CUTEst.jl}: {J}ulia's {CUTEst} interface},
	month = {October},
	howpublished = {\url{https://github.com/JuliaSmoothOptimizers/CUTEst.jl}},
	year = {2020},
}

@book{bertsekas_nonlinear_book,
  title={Nonlinear programming},
  author={Bertsekas, Dimitri P.},
  year={2016},
  edition={third},
  publisher={Athenta Scientific},
  address={Belmont, Massachusetts}
}

@article{curtis2021trust,
	title={Trust-region {N}ewton-{CG} with strong second-order complexity guarantees for nonconvex optimization},
	author={Curtis, Frank E and Robinson, Daniel P and Royer, Cl{\'e}ment W and Wright, Stephen J},
	journal={SIAM Journal on Optimization},
	volume={31},
	number={1},
	pages={518--544},
	year={2021},
	publisher={SIAM}
}

@article{sorensen1982newton,
	title={{N}ewton’s method with a model trust region modification},
	author={Sorensen, Danny C},
	journal={SIAM Journal on Numerical Analysis},
	volume={19},
	number={2},
	pages={409--426},
	year={1982},
	publisher={SIAM}
}

@article{royer2018complexity,
	author = {Royer, Cl{\'e}ment W and Wright, Stephen J},
	journal = {SIAM Journal on Optimization},
	number = {2},
	pages = {1448--1477},
	publisher = {SIAM},
	title = {Complexity analysis of second-order line-search algorithms for smooth nonconvex optimization},
	volume = {28},
	year = {2018}}

@article{armijo1966minimization,
	author = {Armijo, Larry},
	journal = {Pacific Journal of mathematics},
	number = {1},
	pages = {1--3},
	publisher = {Mathematical Sciences Publishers},
	title = {Minimization of functions having {L}ipschitz continuous first partial derivatives},
	volume = {16},
	year = {1966}}

@book{nesterov2013introductory,
	author = {Nesterov, Yurii},
	publisher = {Springer Science \& Business Media},
	title = {Introductory {L}ectures on {C}onvex {O}ptimization: A {B}asic {C}ourse},
	volume = {87},
	year = {2013},
	address={New York}}

@article{carmon2017loweri,
	author = {Carmon, Yair and Duchi, John C and Hinder, Oliver and Sidford, Aaron},
	journal = {Mathematical {P}rogramming},
	title = {Lower bounds for finding stationary points {I}},
	year = {2020}}

@book{nemirovskii1983problem,
	author = {Nemirovskii, Arkadii and Yudin, David Borisovich},
	publisher = {Wiley},
	title = {Problem {C}omplexity and {M}ethod {E}fficiency in {O}ptimization},
	year = {1983}}

@book{conn2000trust,
	author = {Conn, Andrew R and Gould, Nicholas IM and Toint, Philippe L},
	publisher = {SIAM},
	title = {Trust region methods},
	year = {2000},
	address={Philadelphia}}

@article{grapiglia2017regularized,
	title={Regularized {N}ewton methods for minimizing functions with {H}\"{o}lder continuous {H}essians},
	author={Grapiglia, Geovani Nunes and Nesterov, Yu},
	journal={SIAM Journal on Optimization},
	volume={27},
	number={1},
	pages={478--506},
	year={2017},
	publisher={SIAM}
}

@article{cartis2019universal,
	title={Universal regularization methods: varying the power, the smoothness and the accuracy},
	author={Cartis, Coralia and Gould, Nick I and Toint, Philippe L},
	journal={SIAM Journal on Optimization},
	volume={29},
	number={1},
	pages={595--615},
	year={2019},
	publisher={SIAM}
}
